\documentclass[12pt,reqno]{amsproc}

\usepackage[margin=0.95in]{geometry}
\usepackage{bm}
\usepackage{amsmath, amssymb, amsthm}
\usepackage{mathpazo} \renewcommand{\mathbf}{\mathbold}

\renewcommand{\mathbf}{\mathbold}

\newcommand{\flord}{\blacktriangleright}
\usepackage{tikz}
\usetikzlibrary{matrix,arrows}

\usepackage[all,2cell]{xy}
\UseAllTwocells
\usepackage{pb-diagram}
\usepackage{pb-xy}

\usepackage{hyperref}
\hypersetup
{
    colorlinks,	%
    citecolor=red,%
    filecolor=black,%
    linkcolor=blue,%
    urlcolor=blue
}

\usepackage{color}

\theoremstyle{plain}
  \newtheorem{theorem}{Theorem}[section]
  
  \newtheorem{lemma}[theorem]{Lemma}
  
  \newtheorem{proposition}[theorem]{Proposition}

\theoremstyle{definition}
  \newtheorem{definition}[theorem]{Definition}
  
  \newtheorem{ex}[theorem]{Example}
  
  \newtheorem{remark}[theorem]{Remark}
  \newenvironment{example}{\begin{ex}}{\end{ex}}

  \newcommand{\set}[1]{\left\{#1\right\}}

  \newcommand{\R}{\mathbb{R}}

  \newcommand{\category}[1]{\mathrm{\mbox{\bf #1}}}
 
	\newcommand{\Face}{{\category{Fac}}}
	\newcommand{\Entr}{{\category{Ent}}}
	\newcommand{\Rmod}{\category{Mod}(R)}
	
	\newcommand{\bA}{\mathbf{A}}
	\newcommand{\bB}{\mathbf{B}}
	\newcommand{\bC}{\mathbf{C}}
	\newcommand{\bD}{\mathbf{D}}
	\newcommand{\bP}{\mathbf{P}}
	\newcommand{\bE}{\mathbf{E}}
	\newcommand{\bF}{\mathbf{F}}
	\newcommand{\bG}{\mathbf{G}}
	\newcommand{\bI}{\mathbf{I}}
	\newcommand{\bJ}{\mathbf{J}}
	\newcommand{\bL}{\mathbf{L}}
	
	\newcommand{\bN}{\mathbf{N}}
	\newcommand{\bQ}{\mathbf{Q}}

	\newcommand{\bU}{\mathbf{U}}
	\newcommand{\bV}{\mathbf{V}}

	\newcommand{\inv}{{{}^{-1}}}
	\newcommand{\X}{\mathbb{X}}
	\newcommand{\Sph}{\mathbb{S}}
	\newcommand{\bbP}{\mathbb{P}}
	
	\newcommand{\En}[1]{\Entr[#1]} % Exit category
	\newcommand{\Fc}[1]{\Face[#1]} % Face category
	\newcommand{\Fl}[2]{\category{Flo}_{#1}[{#2}]} % Flow category: use as \Fl{matching, complex}
 \newcommand{\Loc}{\category{Loc}}
\newcommand{\Flo}{\category{Flo}}

  \newcommand{\twomor}{\Rightarrow}
	\newcommand{\twomorback}{\Leftarrow}
	
	\newcommand{\inc}{\hookrightarrow}

 \newcommand{\under}{\mathbin{\mkern-5mu/\mkern-8mu/}\mkern-5mu} %euler: 3 6 3
	\newcommand{\fiber}{\under}

\newcommand{\com}{\hspace{-0.12em}\circ\hspace{-0.12em}}

\title{{Discrete Morse Theory and Localization}}
\author{Vidit Nanda}
\date{}

\begin{document}

\begin{abstract}
Incidence relations among the cells of a regular CW complex produce a poset-enriched category of {\em entrance paths} whose classifying space is homotopy-equivalent to that complex. We show here that each acyclic partial matching (in the sense of discrete Morse theory) of the cells corresponds precisely to a homotopy-preserving localization of the associated entrance path category. Restricting attention further to the full localized subcategory spanned by critical cells, we obtain the {\em discrete flow category} whose classifying space is also shown to lie in the homotopy class of the original CW complex. This flow category forms a combinatorial and computable counterpart to the one described by Cohen, Jones and Segal in the context of smooth Morse theory. 
\end{abstract}

\maketitle

\section{Introduction}

To the reader who desires a quick summary of this work, we recommend a brief glance at Figure \ref{fig:deform}. Illustrated there is a regular CW complex $\X$ along with a simple operation which removes two cells $x$ and $y$, where $y$ is a face of $x$ (written $x > y$). 
\begin{figure}[ht]
\includegraphics[scale=0.25]{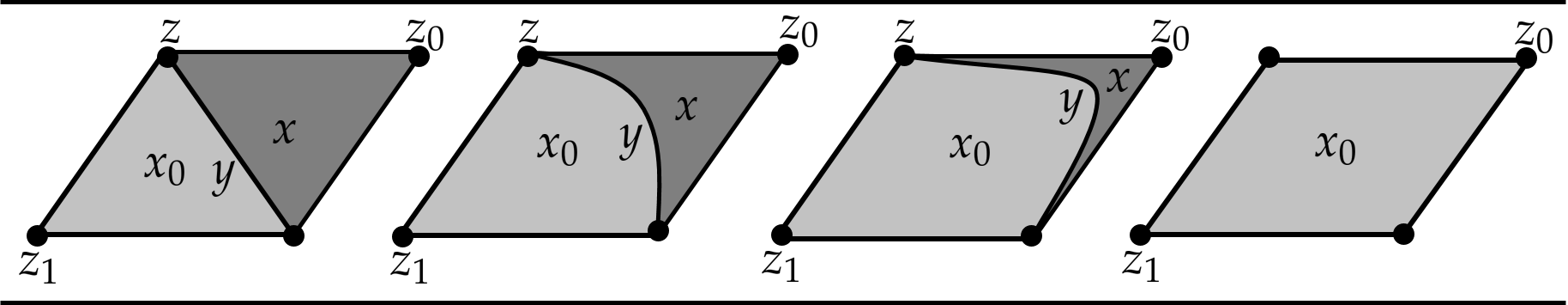}
\caption{Collapse of a cell pair $(x,y)$ in the CW complex $\X$.}
\label{fig:deform}
\end{figure}

Essentially, one drags $y$ across $x$ onto the other cells in the boundary of $x$. If $y$ happens to be a {\em free face} of $x$ (that is, if $x$ is the unique cell of $\X$ containing the closure of $y$ in its boundary), then our operation is an {\em elementary collapse} in the sense of simple homotopy theory \cite{whitehead}. In this special case, it is well-known that one can excise both $x$ and $y$ from $\X$ while preserving both homotopy type and regularity \cite{cohen:73,hatcher:75}. However, it is clear from our figure that if $y$ if {\em not} a free face of $x$, then we must concoct a mechanism to glue other co-faces of $y$ (such as $x_0$) to the remaining faces of $x$ in order to preserve homotopy type once $x$ and $y$ have been removed. Our focus here is on providing an explicit and computable method to perform such attachments.

We immediately sacrifice regularity when pursuing these non-elementary collapses, and therefore must pay careful attention to how the remaining cells are attached. For instance, consider Figure \ref{fig:deform} again and note that $\X$ remains regular even if the vertices $z_0$ and $z_1$ are identified. However, our final complex in this case is not regular precisely because the cell $x_0$ becomes attached to this identified vertex $z_0 \sim z_1$ in two essentially distinct ways\footnote{On the other hand, the attachment of $x_0$ to $z$ remains unaltered across the collapse.}. In order to keep track of such alterations in attaching maps when several collapses are performed, we turn to the {\em entrance paths} \cite{treumann} of $\X$. An entrance path from one cell to another is simply a descending sequence of intermediate faces connecting source to target  --- for instance, $(x > z)$ and $(x > y > z)$ are both entrance paths of $\X$ from $x$ to $z$. The {\bf entrance path category} of $\X$ is a poset-enriched category whose objects are the cells of $\X$, and whose morphisms are entrance paths partially ordered by inclusion, e.g., $(x > z) \twomor (x > y > z)$.  Finite regular CW complexes are homotopy equivalent to the classifying spaces of their entrance path categories\footnote{See Proposition \ref{prop:entpathhom}.}. 

We show here that collapsing the cell pair $x$ and $y$ in $\X$ as described above corresponds to the {\bf localization}, or formal inversion, of the morphism $(x > y)$ in the entrance path category of $\X$. The classifying space of the localized category so obtained is homotopy-equivalent to that of the un-localized one (and hence to $\X$). Moreover, one can safely remove both $x$ and $y$ from the localized category while preserving its homotopy type.

\subsection*{Motivation and related work}

The central purpose of our work is to construct a Morse theory tailored to a class of poset-enriched categories broad enough to contain entrance path categories of all regular CW complexes. Aside from the natural desire to simplify computation of cellular homotopy (and weaker algebraic-topological invariants) by eliminating superfluous cells as described above, we are also motivated by at least two largely disjoint streams of existing results in Morse theory \cite{milnormorse}.

Forman's {\em discrete Morse theory} \cite{forman98} has been successfully used to perform (co)homology computations not only in algebraic topology \cite{curry:ghrist:nanda, mischaikow:nanda, mori:salvetti:11}, but also in commutative algebra \cite{jollenbeck}, topological combinatorics \cite{shareshian}, algebraic combinatorics \cite{sagancpp} and even geometric group theory \cite{brown}. The central idea involves the imposition of a {\em partial matching} $\mu$ on adjacent cell pairs of a regular CW complex $\X$ subject to a global acyclicity condition --- the unmatched cells play the role of critical points whereas the matched cells generate combinatorial gradient-like flow paths. Although it is established that $\X$ is homotopy-equivalent to a CW complex whose cells correspond (in both number and dimension) to the critical cells of $\mu$, there is no explicit description of how these critical cells are attached to each other. A second goal of this paper is to better understand the attaching maps in discrete Morse theory.

On the other hand, the relationship between {\em Morse theory and classifying spaces} in the smooth category has been described by Cohen, Jones and Segal in \cite{cjs}. From a compact Riemannian manifold $\mathbf{X}$ equipped with a (smooth) Morse function $f:\mathbf{X} \to \R$, their work constructs a topologically enriched {\em flow category} $\bC_f$ whose
\begin{itemize} 
\item objects correspond to the critical points of $f$, 
\item morphisms are moduli space of broken gradient flow lines, and 
\item classifying space is homotopy-equivalent to $\mathbf{X}$. 
\end{itemize} Our third goal, then, is to {produce a combinatorial and computable analogue of the flow category from Cohen-Jones-Segal's Morse theory} by replacing Riemannian manifolds and smooth Morse functions by regular CW complexes and acyclic partial matchings. 

While it is tempting to attempt a direct translation of existing smooth arguments to the discrete setting, a fundamental obstacle drives the search for new techniques: in the smooth case, every flow line starting from an arbitrary non-critical point on the manifold terminates at a unique critical point of the Morse function. In sharp contrast, it is an unavoidable consequence of discretization that combinatorial gradient flow paths can split and merge rather viciously. Thus, even a non-critical cell typically admits several gradient paths to many different critical cells, and hence analogues of the smooth techniques are unavailable in this case.

\subsection*{Main results and outline}

Here is (the simplest version of) our main result.

\begin{theorem}
\label{thm:main}
Given a finite regular CW complex $\X$ equipped with an acyclic partial matching $\mu$, let $\Sigma = \{(x_\bullet > y_\bullet)\}$ denote the collection of all entrance paths which correspond to the $\mu$-pairings $\mu(y_\bullet)=x_\bullet$. Then, there exists a poset-enriched category $\Fl{\Sigma}{\X}$ whose objects are the critical cells of $\mu$, whose morphisms consist of the entrance paths of $\X$ localized about $\Sigma$, and whose classifying space is homotopy-equivalent to $\X$.
\end{theorem}

We call $\Fl{\Sigma}{\X}$ the {\bf discrete flow category} associated to the acyclic matching $\mu$ in order to emphasize the analogy with the flow category of Cohen, Jones and Segal from \cite{cjs}. Our strategy is as follows: we examine functors from the entrance path category $\En{\X}$ and the flow category $\Fl{\Sigma}{\X}$ into the localization $\Loc_\Sigma\En{\X}$ of $\En{\X}$ about $\Sigma$:
\[
\xymatrixrowsep{0.5in}
\xymatrixcolsep{0.5in}
\xymatrix{
\En{\X}\ar@{->}[r]^-{\bL_\Sigma}  & \Loc_\Sigma\En{\X} & \Fl{\Sigma}{\X} \ar@{->}[l]_-{\bJ}
}
\]  
and show that both $\bL_\Sigma$ (the canonical localization functor) and $\bJ$ (the inclusion of a full subcategory) induce homotopy-equivalences of classifying spaces. The central tool in both cases is Quillen's {\bf Theorem A} \cite{quillen} adapted to the 2-categorical setting \cite{bcgeonerve}.

A version of the discrete flow category $\Fl{\Sigma}{\X}$ has been constructed in \cite{ntt}, where its classifying space was shown to be homotopy-equivalent to (a $\mu$-dependent regular subdivision of) $\X$ via a {\em collapsing functor}. There are several differences between our model of $\Fl{\Sigma}{\X}$ and the one in \cite{ntt}. Most significantly, the relationship of the flow category to the entrance path category remains unexplored in \cite{ntt}. Our definition here has the advantage of being in a position to easily produce new, general flavors of discrete Morse theory by accessing the universal property of localization. On the other hand, we are unable to obtain a direct map $\En{\X} \to \Fl{\Sigma}{\X}$ analogous to the collapsing functor.

A smooth Morse function on a compact manifold is said to be {\em Morse-Smale} \cite[Ch 4]{banyaga:hurtubise:2004} whenever all stable and unstable manifolds intersect transversely\footnote{In other words, if the unstable manifold $W^-$ of an index-$p$ critical point intersects the stable manifold $W^+$ of an index-$q$ critical point, then their intersection $W^- \cap W^+$ is a manifold of dimension $(p-q-1)$. See \cite[Thm 2.27]{nicolaescu:12} for details.}. In addition to being a generic property of smooth Morse functions, such transversality greatly simplifies several standard arguments; there have, therefore, been intricate efforts to similarly define and exploit transversality for piecewise-linear \cite{akin:72, armstrong:zeeman:67} and topological \cite[Ch III.1]{kirby:siebenmann:77} manifolds. The flow category provides a natural definition of transversality in the context of acyclic partial matchings on regular CW complexes. Namely, a pairing $\Sigma$ is Morse-Smale if the following condition holds across all pairs of critical cells $c$ and $c'$: either the poset $\Fl{\Sigma}{\X}(c,c')$ of morphisms from $c$ to $c'$ in the flow category is empty, or its classifying space is a manifold of dimension $(\dim c - \dim c' - 1)$. We intend to carefully explore cellular Morse-Smale transversality in detail elsewhere.

The rest of this paper is organized as follows. Section \ref{sec:2cats} contains background material regarding poset-enriched categories and discrete Morse theory. Section \ref{sec:entpath} describes the entrance path category associated to each regular CW complex and catalogs some of its relevant properties. In Section \ref{sec:lochom} we show that the localization functor $\bL_\Sigma$ induces a homotopy-equivalence of classifying spaces by appealing to the 2-categorical Theorem A. Section \ref{sec:flocat} introduces the discrete flow category and establishes that its inclusion $\bJ$ into the localized entrance path category also induces a homotopy-equivalence. In Section \ref{sec:calc} we explicitly compute the discrete flow category in three different cases. In Section \ref{sec:newdmt} we record the general result (Theorem \ref{thm:gen}) which implies Theorem \ref{thm:main} and describe how one may use it to construct discrete Morse theories in broader contexts. 

\section{Homotopy and localization for p-categories} \label{sec:2cats}

We will work exclusively with small categories enriched over posets, which are called {\bf p-categories} throughout this paper. For a general treatment of enriched categories, see \cite{kelly}.

A p-category $\bC$ consists of a set $\bC_0$ of objects, and between each pair $x$ and $y$ of such objects there is a (possibly empty) poset $\bC(x,y)$ of morphisms subject to standard axioms, as described below. We write $f:x \to y$ to indicate that $f$ is an element of $\bC(x,y)$ and $f \twomor g$ to indicate that $f$ is less than $g$ as an element of $\bC(x,y)$. Moreover, one requires
\begin{itemize} 
\item for each object $x \in \bC_0$ the existence of a distinguished {\em identity} $1_x$ in $\bC(x,x)$, and 
\item across each triple $x,y,z$ of objects the presence of a {\em composition} \[\circ_{xyz}: \bC(x,y) \times \bC(y,z) \to \bC(x,z)\] which is associative, respects identities, and preserves the partial orders induced by $\twomor$ on its domain and codomain.
\end{itemize} When all morphism-sets of $\bC$ are given the trivial partial order, one recovers ordinary 1-categories; and since every poset is a 1-category (with at most one morphism between any pair of objects), every p-category is automatically a 2-category (see \cite{kelly}). Thus, p-categories lie properly between 1-categories and 2-categories.

A {\bf p-functor} from $\bC$ to another p-category $\bD$, written $\bF:\bC \to \bD$, assigns an object $\bF x$ of $\bD_0$ to  each object $x$ of $\bC_0$; and each poset $\bC(x,y)$ is mapped monotonically to the corresponding $\bD(\bF x,\bF y)$. We require all p-functors in sight to preserve identities (this property is often called {\em normality} \cite{bcgeonerve}); and given a pair of composable morphisms $f$ and $g$ in $\bC$, we require $\bF(f \com g) \twomor \bF(f) \com \bF(g)$ to hold in $\bD$. Such functors are called {\em oplax} in most references, and if the relation above is always an equality then the functor in question is {\em strict}. Since every strict p-functor is automatically oplax, we will describe and use results for oplax functors even though most functors which appear in our main arguments are strict.

Composition of p-functors is defined in the usual manner. An oplax p-natural transformation $\eta:\bF \twomor \bG$ between (strict or oplax) p-functors $\bF,\bG:\bC \to \bD$ assigns to each object $x$ of $\bC$ a morphism $\eta_x:\bF x \to \bG x$ in $\bD$ so that the following order relation holds for each $f:x \to y$ in $\bC$: 
\[
\xymatrixcolsep{0.4in}
\xymatrixrowsep{0.4in}
\xymatrix{%
	\bF x \ar@{->}[d]_{\bF f} \ar@{->}[r]^{\eta_x} & \bG x \ar@{->}[d]^{\bG f} \\
	\bF y \ar@{->}[r]_{\eta_y} \ar@{=>}[ur]^{} & \bG y
	}
\]
If the order relation depicted above is an equality for every $f$, then $\eta$ is called strict.

Given a p-category $\bC$, the {\bf full subcategory} $\bC'$ spanned by a subset $\bC'_0 \subset \bC_0$ has 
\begin{itemize}
\item $\bC'_0$ as its set of objects, and
\item $\bC'(x,y) = \bC(x,y)$ as the poset of morphisms from $x$ to $y$,
\end{itemize}
with composition being inherited verbatim from $\bC$. There is an obvious (strict) inclusion p-functor $\bC' \inc \bC$ in this case.

\begin{definition}
\label{def:atom}
A morphism $f:x \to y$ in a p-category $\bC$ is called an {\bf atom} if
\begin{enumerate}
\item $f \twomor f'$ holds for any $f' \in \bC(x,y)$,
\item $x = y$ implies $f = 1_x$, and
\item solutions to $g \com h \twomor f$ for morphisms $g:x \to z$ and $h:z \to y$ only exist 
			\begin{itemize}
				\item when $z = x$, in which case $(g,h) = (1_x,f)$, or 
				\item when $z = y$, in which case $(g,h) = (f,1_y)$;
			\end{itemize}
\end{enumerate}
In other words, atoms are simultaneously {\em minimal} and {\em weakly indecomposable} morphisms.
\end{definition}

\subsection{Nerves, fibers and homotopy}
\label{ssec:nerfibhom}
Small 1-categories have a homotopy theory arising from the {\bf nerve} construction: to each category one canonically associates a simplicial set \cite{friedsset, maysset} whose vertices are the objects and whose $n$-simplices correspond to sequences of $(n+1)$ composable morphisms. The following notion (adapted from \cite{bcgeonerve}) provides one model for extending the classifying space construction to p-categories. Although there are (at least) ten reasonable models for nerves of bicategories (which subsume p-categories), the main result of \cite{carrasco:cegarra:garzon:10} confirms that all are equivalent up to homotopy. 

\begin{definition} 
\label{def:geonerve}
The (oplax) {\bf geometric nerve} $\Delta\bC$ of a p-category $\bC$ is that simplicial set whose vertices are the objects $\bC_0$, and whose $n$-simplices spanning objects $x_0, \ldots, x_n$ consist of morphisms $f_{ij}:x_i \to x_j$ satisfying $f_{ik} \twomor  f_{ij} \com f_{jk}$ for all $0 \leq i \leq j \leq k \leq n$, with the understanding that $f_{ii} = 1_{x_i}$ for all $i$.
\end{definition}

Each edge of $\Delta\bC$ across vertices $x,y \in \bC_0$ corresponds to a morphism $f:x \to y$. And there is a unique 2-simplex across edge $f:x \to y$, edge $g: y \to z$ and edge $h: x \to z$ if and only if the relation $h \twomor f \circ g$ holds in $\bC(x,z)$. A simplex of dimension exceeding two exists in $\Delta\bC$ if and only if all of its 2-dimensional faces are present. Thus, the geometric nerve is {\em co-skeletal} beyond dimension 2.

It is readily confirmed that any strict or oplax p-functor $\bF:\bC \to \bD$ prescribes a simplicial map $\Delta \bF: \Delta \bC \to \Delta \bD$ of geometric nerves, and it is standard to ask when this map induces a homotopy-equivalence of classifying spaces. The 1-categorical version of the following result is an immediate consequence of \cite[Prop 2.1]{segalcsss}, while the avatar presented below follows from a bicategorical generalization \cite[Lem 2.6]{cegarra}. Here $\bI_\bC:\bC \to \bC$ is the identity p-functor. 

\begin{theorem}
\label{thm:nattranshomeq}
Let $\bF:\bC \to \bD$ and $\bG:\bD \to \bC$ be oplax p-functors. If there exist oplax natural transformations $\eta: \bI_\bC  \twomor \bF\bG$ and $\nu: \bG\bF \twomor \bI_\bD$, then $\Delta\bF$ and $\Delta\bG$ are homotopy-inverses.
\end{theorem}

It is often useful to identify when certain p-categories have contractible classifying spaces (i.e., those possessing the homotopy type of a point) --- recall, for instance, that 1-categories with initial or terminal objects are contractible\footnote{An {\em initial object} in a 1-category is any object which admits precisely one morphism to every object in that category (including itself) while a {\em terminal object} is one which admits only one morphism from every object. In the special case where our category is a poset, initial and terminal objects correspond to minimal and maximal elements respectively.}. The following definition provides up-to-homotopy analogues of such objects in p-categories.

\begin{definition}
\label{def:homaxmin}
An object $z$ of a p-category $\bC$ is called {\bf homotopy-maximal} if each poset $\bC(w,z)$ for $w \in \bC_0$ has a maximal element $f_w$, with $f_z = 1_z$. Similarly, the object $w$ is {\bf homotopy-minimal} if for each object $z \in \bC_0$ the poset $\bC(w,z)$ has a minimal element $f_z$, with $f_w = 1_w$.
\end{definition} 
We will appeal to the following result with considerable frequency.
\begin{lemma}
\label{lem:qinit}
Any p-category $\bC$ containing a homotopy-maximal or homotopy-minimal object has a contractible classifying space.
\end{lemma}
\begin{proof}
Let $z$ be a homotopy-maximal object of $\bC$, so for each $w \in \bC_0$ there is a maximum element $f_w: w \to z$ in $\bC(w,z)$ (with $f_z = 1_z$). Let $\bullet$ be the trivial p-category consisting of the single object and only the identity morphism. Consider two p-functors which relate $\bC$ and $\bullet$ --- there is a p-functor $\bullet \hookrightarrow \bC$ given by sending the unique object of $\bullet$ to $z$, and there is an obvious surjection $\bC \twoheadrightarrow \bullet$. Composing the inclusion with the surjection in one direction immediately produces the identity p-functor on $\bullet$. Composing in the opposite direction yields a strict p-functor $\bC \to \bC$ which sends every object to $z$ and every morphism to $1_z$. We claim that this p-functor admits an oplax p-natural transformation from the identity functor on $\bC$: send each $w \in \bC_0$ to $f_w: w \to z$, and each morphism $g:w \to w'$ in $\bC$ to the order relation $g \com f_{w'} \twomor f_w$ which results from the fact that $f_w$ is maximal in $\bC(w,z)$. In short, we have the following diagram:
\[
\xymatrixcolsep{0.4in}
\xymatrixrowsep{0.4in}
\xymatrix{
w \ar@{->}[d]_{g} \ar@{->}[r]^{f_w} & z \ar@{->}[d]^{1_z}  \\
w' \ar@{->}[r]_{f_{w'}}  \ar@{=>}[ur]^{} & z 
}
\]
Theorem \ref{thm:nattranshomeq} now yields a homotopy-equivalence of $\bC$ with $\bullet$, and hence establishes that the classifying space $|\Delta\bC|$ is contractible. A very similar argument (involving hitherto-unmentioned {\em lax} p-functors and natural transformations) furnishes contractibility in the presence of a homotopy-minimal object.
\end{proof}

When seeking to establish that a p-functor $\bF:\bC \to \bD$ induces homotopy-equivalence  in the absence of a p-functor going back from $\bD$ to $\bC$, one typically resorts to a contractible fiber argument.

\begin{definition}
\label{def:fiber2cat}
Given an oplax p-functor $\bF:\bC \to \bD$ and an object $z \in \bD$, the {\bf fiber} of $\bF$ over $z$, denoted $\bF \fiber z$, is a p-category whose 
\begin{enumerate}
\item objects are pairs $(w,g)$ where $w \in \bC_0$ and $g \in \bD(\bF w,z)$, and
\item morphisms from $(w,g)$ to $(w',g')$ are given by $h \in \bC(w,w')$ satisfying $\bF h \com g' \twomor g$, 
\end{enumerate}
where the partial order on morphisms is contravariant to the one inherited from $\bC$ --- namely, $h \twomor h'$ in $\bF \fiber z$ if $h' \twomor h$ in $\bC$.
\end{definition}
The rules for composition of morphisms of fiber categories are inherited from $\bC$ and $\bD$ in a reasonably straightforward manner, so we refer the curious reader to \cite{bcgeonerve} for details and instead illustrate a morphism $h:(w,g) \to (w',g')$ in $\bF \fiber z$:
\[
\xymatrixcolsep{0.25in}
\xymatrixrowsep{0.35in}
\xymatrix{
\bF w \ar@{->}[rr]_*+<0.8em>{\stackrel{}{\Longleftarrow}}^{\bF h} \ar@{->}[dr]_-{g} & & \bF w' \ar@{->}[dl]^-{g'} \\
 & z & 
}
\]
The following result follows directly from \cite[Thm 2]{bcgeonerve} and generalizes Quillen's {\bf Theorem A} \cite{quillen} to p-categories. In particular, it provides sufficient conditions which guarantee that a p-functor induces homotopy-equivalence of classifying spaces.
\begin{theorem}[Theorem A for p-categories]
\label{thm:quillenA}
Let $\bF:\bC \to \bD$ be a (strict or oplax) p-functor. If the classifying space $|\Delta(\bF\fiber z)|$ of the fiber over each $z \in \bD_0$ is contractible, then $\Delta\bF: \Delta\bC \to \Delta\bD$ induces a homotopy-equivalence.
\end{theorem}

\subsection{Localization}
\label{subsec:loc}

Given a (small, ordinary) category  $\bC$ and a collection $\Sigma$ of its morphisms that contains all identies and is closed under composition, one can define the {\bf localization of $\bC$ about $\Sigma$} --- this is the minimal category $\bC[\Sigma\inv]$ containing $\bC$ where members of $\Sigma$ have been formally inverted \cite{dwyerkan2, dwyerkan, gabrielzisman}. The localization comes with a functor $\bL: \bC \to \bC[\Sigma\inv]$ characterized by the following universal property: every morphism in $\Sigma$ is sent to an isomorphism by $\bL$, and every other functor $\bF:\bC \to \bD$ which sends $\Sigma$-elements to isomorphisms in $\bD$ admits a Kan extension $\bF'$ across $\bL$. In other words, the triangle below commutes:
\[
\xymatrixrowsep{0.3in}
\xymatrixcolsep{0.2in}
\xymatrix{
\bC \ar@{->}[rr]^{\bF} \ar@{->}[dr]_-{\bL} & & \bD\\
& \bC[\Sigma\inv] \ar@{-->}[ur]_-{\bF'} & 
}
\]

Here we describe localization for p-categories about a special class of morphisms. Call a collection $\Sigma$ of morphisms in a p-category {\em directed} if it only contains atoms, and if for each $f:x \to y$ in $\Sigma$, we have both $x \neq y$ and the absence of morphisms from $y$ to $x$ in $\Sigma$.
\begin{definition}
\label{def:loc2cat}
Let $\bC$ be a p-category and $\Sigma$ a directed collection of its morphisms so that the union $\Sigma^+$ of $\Sigma$ with all identities is closed under composition. The {\bf localization of $\bC$ about $\Sigma$}, written $\Loc_\Sigma\bC$, is a p-category given by the following data.
\begin{enumerate}
\item The set $\Loc_\Sigma\bC_0$ of objects is precisely the same as $\bC_0$.
\item Given objects $w,z \in \bC_0$, every morphism $\gamma : w \to z$ in $\Loc_\Sigma\bC$ is an equivalence classes of finite but arbitrarily long {\em $\Sigma$-zigzags} in $\bC$ of the form
\[
\xymatrixrowsep{0.45in}
\xymatrixcolsep{0.35in}
\xymatrix{
w \ar@{->}[r]^{g_0} & y_0 & x_0 \ar@{->}[l]_{f_0} \ar@{->}[r]^{g_1} & y_1 & \cdots \ar@{->}[l]_{f_1} &  x_k \ar@{->}[l]_{f_k} \ar@{->}[r]^{g_{k+1}} & z
}
\]
where the left-pointing $f_\bullet$ are to be chosen from $\Sigma^+$, the right-pointing $g_\bullet$ are arbitrary, and the equivalence is generated by the following relations. Two zigzags are related
\begin{itemize}
\item {\bf horizontally} if they differ by intermediate identity maps, or 
\item {\bf vertically} if they form the rows of a commuting diagram in $\bC$:
\[
\xymatrixcolsep{0.5in}
\xymatrixrowsep{0.45in}
\xymatrix{
w \ar@{->}[r]^{g_0} \ar@{=}[d]_{} &  y_0  \ar@{->}[d]^{u_0} & x_0 \ar@{->}[l]_{f_0} \ar@{->}[r]^{g_1} \ar@{->}[d]^{v_0} & \cdots & x_{k} \ar@{->}[r]^{g_{k+1}} \ar@{->}[l]_{f_k} \ar@{->}[d]_{v_{k}} & z \ar@{=}[d] \\
w \ar@{->}[r]_{g'_0} \ar@{=}[ur]^{}	&  y'_0 	& x'_1 \ar@{->}[l]^{f'_0}  	\ar@{->}[r]_{g'_1} \ar@{=}[ul]_{} \ar@{=}[ur]_{} & \cdots & x'_{k} \ar@{->}[l]^{f'_k} \ar@{->}[r]_{g'_{k+1}} \ar@{=}[ul]^{} \ar@{=}[ur]_{}& z
}
\]
where all vertical morphisms $u_\bullet$ and $v_\bullet$ also lie in $\Sigma^+$.
\end{itemize}

\item The partial order on morphisms in $\Loc_\Sigma\bC$ is obtained by replacing all equalities by order relations in the commuting diagram of (2):
\[
\xymatrixcolsep{0.5in}
\xymatrixrowsep{0.45in}
\xymatrix{
w \ar@{->}[r]^{g_0} \ar@{=>}[d]_{} &  y_0  \ar@{->}[d]^{u_0} & x_0 \ar@{->}[l]_{f_0} \ar@{->}[r]^{g_1} \ar@{->}[d]^{v_0} & \cdots & x_{k} \ar@{->}[r]^{g_{k+1}} \ar@{->}[l]_{f_k} \ar@{->}[d]_{v_{k}} & z \ar@{=}[d] \\
w \ar@{->}[r]_{g'_0} \ar@{=>}[ur]^{}	&  y'_0 	& x'_1 \ar@{->}[l]^{f'_0}  	\ar@{->}[r]_{g'_1} \ar@{=>}[ul]_{} \ar@{=>}[ur]_{} & \cdots & x'_{k} \ar@{->}[l]^{f'_k} \ar@{->}[r]_{g'_{k+1}} \ar@{=>}[ul]^{} \ar@{=>}[ur]_{}& z
}
\]
Thus, $\gamma' \twomor \gamma$ in $\Loc_\Sigma\bC(w,z)$ if and only if there exist $\Sigma$-zigzags representing $\gamma$ and $\gamma'$ which fit into the top and bottom rows (respectively) of a  diagram like the one above. Again, the vertical morphisms $u_\bullet$ and $v_\bullet$ must lie in $\Sigma^+$.

\item Composition of morphisms in $\Loc_\Sigma\bC$ is given by concatenating representative $\Sigma$-zigzags, with the understanding that the last right-pointing map of the first morphism is to be composed in $\bC$ with the first right-pointing map of the second morphism.    
\end{enumerate}
\end{definition}

One can check (using the directedness of $\Sigma$) that $\twomor$ is indeed a partial order, and that the composition $\circ$ is well-defined and order-preserving. While the construction above may be slightly more involved than ordinary 1-categorical localization (from \cite[Ch 1]{gabrielzisman}), it is less intricate than the {\em hammock localization} of 1-categories from \cite{dwyerkan2} and the localization of enriched categories in general (as in \cite{wolff:73}). In any event, the reader can find three examples of localization in Section \ref{sec:calc}.
\begin{remark}
\label{rmk:reductions}
If a right-pointing map in some $\Sigma$-zigzag factors as $g\circ f$ where $f:x \to y$ lies in $\Sigma$, and if it is followed by a left-pointing $f$, then one can simplify to a shorter $\Sigma$-zigzag which represents the same morphism in $\Loc_\Sigma\bC$ as follows: 
\[
\left(\cdots \stackrel{g\circ f}{\longrightarrow} y \stackrel{f}{\gets} x \stackrel{h}{\to} \cdots\right) \sim 
\left(\cdots \stackrel{g}{\to} x \stackrel{1}{\gets} x \stackrel{h}{\to} \cdots\right) \sim
\left(\cdots \stackrel{g\circ h}{\longrightarrow} \cdots\right)
\]
Here the first equivalence follows from {\em vertical} reduction and the second from {\em horizontal} reduction as described in the preceding definition. In this sense, a left-pointing $f \in \Sigma$ cancels the preceding right-pointing $f$.
\end{remark}

The associated localization p-functor $\bL_\Sigma:\bC \to \Loc_\Sigma\bC$ is strict, and essentially given by inclusion --- each object is mapped identically to itself, and each morphism is sent to its own equivalence class of zigzags. Any other (strict) p-functor $\bC \to \bD$ which sends the morphisms in $\Sigma$ to isomorphisms in $\bD$ admits a Kan extension across $\bL_\Sigma$.

\subsection{Discrete Morse theory} \label{sec:DMT}

Discrete Morse theory is a combinatorial adaptation of Morse theory \cite{forman98, knudson}. The underlying engine which powers the main results is the notion of simple homotopy equivalence \cite{cohen:73}. Let $\X$ be a finite regular CW complex. We write $y \lhd x$ to indicate that the cell $y$ is a co-dimension 1 face of the cell $x$ (that is, $\dim x - \dim y = 1$). 

\begin{definition}
\label{def:partmatch}
An {\bf acyclic partial matching} on $\X$ consists of a partition of the cells into three disjoint sets $D, U$ and $M$ along with a bijection $\mu:D \to U$ so that the following conditions hold.
\begin{enumerate}
\item {\bf Incidence:} $d \lhd \mu(d)$ for each $d \in D$, and
\item {\bf Acyclicity:} the transitive closure of the binary relation 
\[
d \prec_\mu d' \text{ if and only if } d \lhd \mu(d')
\]
is a partial order on $D$.
\end{enumerate}
The unpaired cells which lie in $M$ are called {\bf critical} cells of $\mu$ in analogy with smooth Morse theory. 
\end{definition}

\begin{figure}[h!]
\includegraphics[scale=0.2]{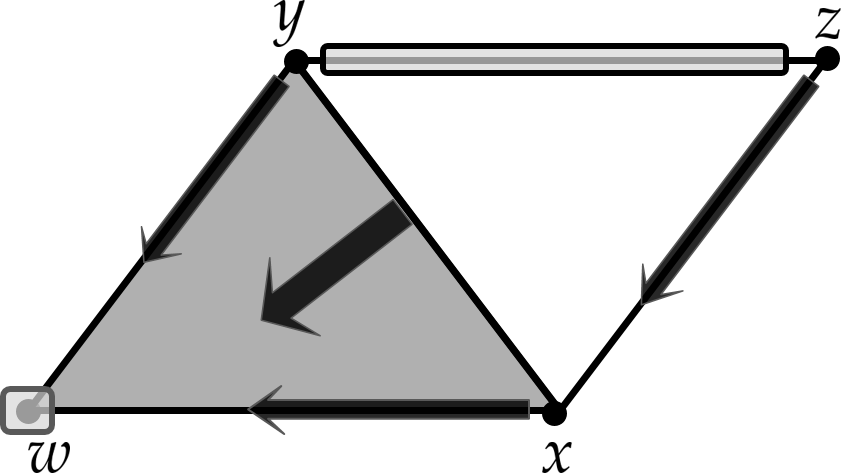}
\caption{An acyclic partial matching $\mu$ on a small simplicial complex. Matched pairs are shown via arrows $d \to \mu(d)$, while critical cells are boxed.}
\label{fig:match}
\end{figure}

Figure \ref{fig:match} illustrates a simple acyclic partial matching $\mu$ on a finite simplicial complex which lies in the homotopy class of a circle. One may extend $\prec_\mu$ to a partial order on $D \cup M$ as follows: each $m \in M$ is strictly smaller than every $m' \in M$ or $d \in D$ in its co-boundary and strictly larger than every $m'$ in its boundary. Moreover, we declare $m \prec_\mu d$ for $d \in D$ whenever $m$ is a face of $\mu(d)$.  Recall that a {\em linear extension} of $\prec_\mu$ is a well-ordering of the cells in $M \cup D$ which is monotone with respect to $\prec_\mu$. Given any such extension $\{e_1,\ldots,e_J\}$ of $\prec_\mu$ consisting of cells from $M \cup D$, let $\X^j$ for $j \leq J$ be the {\em reduced} CW subcomplex of $\X$ defined as follows. It contains all the cells of $\X$ {except} the union of $e_i$ (and $\mu$-paired cells in $U$, if any) across all $i > j$. The following results of Forman \cite[Thm 3.3 and 3.4]{forman98} mimic the traditional smooth Morse lemmas  \cite[Thm 3.1 and 3.2]{milnormorse}. 

\begin{lemma}
\label{thm:dmlemma}
Let $\mu:D \to U$ be an acyclic partial matching on $\X$ with critical cells $M$. Let $\{e_1, e_2, \cdots,e_J\}$ be any ordering of $M \cup D$ which forms a linear extension of $\prec_\mu$, and let $\X^j$ be the corresponding reduced CW complexes.
\begin{enumerate}
\item[\bf A.] If $\{e_i, e_{i+1}, \ldots,e_j\}$ contains no critical cells, then $\X^j$ simple-homotopy collapses onto $\X^i$ via the removal of pairs $(\mu(e_k),e_k)$ in descending order for $k$ between $i$ and $j$.
\item[\bf B.] If $\{e_i, e_{i+1}, \ldots, e_j\}$ contains exactly one critical cell (of dimension $n$), then $\X^j$ is homotopy-equivalent to $\X^i$ with a single $n$-dimensional disk attached along its boundary.
\end{enumerate}
\end{lemma}
As mentioned in the Introduction, it is important to note that there is no explicit description of how the critical cells are attached along their boundaries --- the proof of Lemma {\bf B} follows from an inductive argument which relies on the finiteness of $\X$, and in particular it is straightforward to construct simple examples where the attaching maps are not regular. Even so, the following main theorem of discrete Morse theory is an immediate consequence of Lemmas {\bf A} and {\bf B}.
\begin{theorem}
\label{thm:dmhomotopy}
Let $\mu:D \to U$ be an acyclic partial matching on $\X$ with critical cells $M$. Then, $\X$ is homotopy-equivalent to a CW complex whose $n$-dimensional cells correspond bijectively with the $n$-dimensional cells in $M$.
\end{theorem}

Turning once again to Figure \ref{fig:match}, we note that one possible linearization of $\prec_\mu$ is given by the sequence $(w,x,y,xy,z,yz)$. In the interval between $w$ and $z$, the simple homotopy type coincides with that of a point; in fact, a sequence of collapses to $w$ is given by following the arrows. Attaching the critical cell $yz$ via its boundary changes the homotopy type to that of a circle (and hence recovers the homotopy class of our original simplicial complex) --- the attaching map is not regular, since both ends of the edge $yz$ are mapped to the vertex $w$ via the paths
\[
yz \rhd y \lhd wy \rhd w \quad \text{ and } \quad yz \rhd z \lhd xz \rhd x \lhd wx \rhd w.
\]
Unlike in the case of smooth Morse theory, there is no unique path from a non-critical cell to a critical cell. In particular, there are two distinct gradient paths from $xy$ to $w$.
%
%It is known \cite[Sec 7 and 8]{forman98} that associating an algebraic index (i.e., a single integer) to each such {\em gradient path} produces a chain complex with chain groups freely generated by the critical cells and boundary operator induced by sums-over-paths of indices. More importantly, this smaller chain complex is quasi-isomorphic to the one which recovers the integer homology of the underlying CW complex.  Several homological algorithms \cite{harker:mischaikow:mrozek:nanda, mischaikow:nanda, curry:ghrist:nanda} and software packages \cite{per} now construct a much smaller {\bf Morse chain complex} before resorting to the relatively expensive matrix algebra which then computes cellular homology. Subsequent efforts \cite{batzies:welker:00, skoldberg} have assigned more refined algebraic data (i.e., a linear map) to each such gradient path, and thus recovered the homology of $\X$ with coefficients valued in constructible cosheaves. 

\section{Cellular categories and Morse systems} \label{sec:entpath}

Given cells $x$ and $y$ of a finite regular CW complex $\X$, we write $x > y$ to denote the strict face relation indicating that the closure of $y$ is contained in the boundary of $x$. Recall that the {\em face poset} $\Fc{\X}$ consists of cells as objects, identity morphisms, and a unique morphism $x \to y$ whenever $x > y$. The nerve of this face poset is the first barycentric subdivision of $\X$, so in particular $\X$ is homotopy-equivalent \cite[Thm 1.7]{lw} (and indeed, homeomorphic) to $|\Delta \Fc{\X}|$.

\subsection{Entrance path categories}

The following p-category \cite{treumann} is a thicker version of $\Fc{\X}$ --- the partial order on its morphisms is designed to catalog {\em how} one cell lies in the boundary of another.

\begin{definition}
\label{def:ent2cat}
The {\bf entrance path category} $\En{\X}$ of (a regular CW complex) $\X$ is the p-category given by the following data.
\begin{enumerate}
\item The objects $\En{\X}_0$ are the cells of $\X$.

\item For each $x,y$ in $\En{\X}_0$, the set $\En{\X}(x,y)$ of morphisms has as its objects the {\em entrance paths}, which are strictly descending sequences $f = (x = x_0 > \cdots > x_k = y)$ of cells from $x$ to $y$.

\item Given $f$ and $f'$ in $\En{\X}(x,y)$, we have the order relation $f \twomor f'$ in $\En{\X}(x,y)$ if and only $f$ is a (not necessarily contiguous) subsequence of $f'$.

\item If $g \in \En{\X}(y,z)$ is the entrance path $(y = y_0 > \cdots > y_\ell = z)$, then its composite with $f:x \to y$ from (2) is given by concatenation, i.e.,
\[
f \com g = (x = x_0 > \cdots > x_k = y = y_0 > \cdots > y_\ell = z) \in \En{\X}(x,z). 
\]
It is easy to check that this composition is order-preserving and that the unique element $(x)$ of $\En{\X}(x,x)$ functions as the identity.
\end{enumerate}
\end{definition}

\begin{figure}[ht]
\includegraphics[scale=0.2]{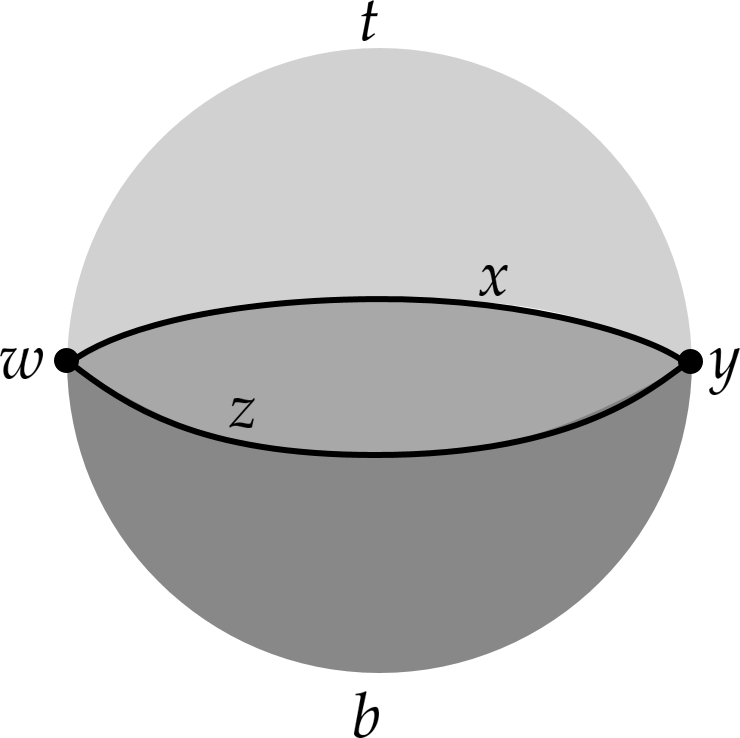}
\caption{A minimal regular CW decomposition of the 2-sphere.}
\label{fig:cwsphere}
\end{figure}

\begin{example}
Let $\Sph$ be the minimal regular CW decomposition of the 2-sphere as shown in Figure \ref{fig:cwsphere}. The elements of $\En{\Sph}_0$ are the cells $\set{w,x,y,z,t,b}$. The poset of morphisms $\En{\Sph}(t,x)$ consists only of the sequence $(t > x)$, whereas $\En{\Sph}(t,w)$ is the poset
\[
(t>x>w) \twomorback (t>w) \twomor (t>z>w).
\]
The composite of $(t > x)$ in $\En{\Sph}(t,x)_0$ and $(x > w)$ in $\En{\Sph}(x,w)_0$ equals $(t > x > w)$, which is an object of $\En{\Sph}(t,w)$.
\end{example}

\begin{proposition}
\label{prop:entpathhom}
If $\X$ is a finite regular CW complex, then there is a homotopy-equivalence 
\[
\X \sim |\Delta \En{\X}|,
\]
between $\X$ and the classifying space of its entrance path category.
\end{proposition}
\begin{proof}
The result is a standard application of Theorem \ref{thm:nattranshomeq} --- we show that $\En{\X}$ is homotopy-equivalent to the face poset $\Fc{\X}$ and hence to the first barycentric subdivision of $\X$. Note that $\Fc{\X}$ includes into $\En{\X}$ since the object-sets coincide and the face relation $x > y$ , whenever it holds, constitutes the minimal entrance path in the poset $\En{\X}(x,y)$. Consider the lax p-functor $\bP:\En{\X} \to \Fc{\X}$ which fixes all the objects and sends each entrance path $(x > \cdots > y)$ to this minimal path $(x > y)$. On one hand, including $\Fc{\X}$ into $\En{\X}$ and then mapping back to $\Fc{\X}$ via $\bP$ yields the identity p-functor on $\Fc{\X}$. On the other hand, for each pair of cells $x$ and $y$ in $\X$ with $x > y$ and any entrance path $(x > \cdots > y)$, we have:
\[
\xymatrixcolsep{0.4in}
\xymatrixrowsep{0.4in}
\xymatrix{%
	x \ar@{->}[d]_{(x > y)} \ar@{->}[r]^{(x)} & x \ar@{->}[d]^{(x > \cdots > y)} \\
	y \ar@{->}[r]_{(y)} \ar@{=>}[ur]^{} & y
	}
\]
Thus, $\bP$ followed by the inclusion of $\Fc{\X}$ into $\En{\X}$ admits an oplax natural transformation to the identity p-functor on $\En{\X}$ and so by Theorem \ref{thm:nattranshomeq} there is a homotopy-equivalence of classifying spaces $|\Delta \En{\X}| \sim |\Delta\Fc{\X}|$ as desired.
\end{proof}

\subsection{Morse systems on cellular categories}

Our goal in this section is to generalize entrance path categories and axiomatize the desirable properties of acyclic partial matchings on regular CW complexes. The following definition highlights a class of p-categories which is simultaneously broad enough to include all entrance path categories and narrow enough to admit a discrete Morse theory.

\begin{definition}
\label{def:cell2cat}
A p-category $\bE$ is called {\bf cellular} if for each $x, y \in \bE_0$, the poset $\bE(x,y)$ is either empty or it contains an atom (in the sense of Definition \ref{def:atom}).
\end{definition}
 It is straightforward to check that the entrance path category of any regular CW complex is cellular, since the minimal entrance path between any pair of adjacent cells $x \geq y$ is the atom of $\En{\X}(x,y)$.

Next, we turn to the analogue of Morse functions in our setting. By Definition \ref{def:partmatch}, every acyclic partial matching $\mu$ on a regular CW complex $\X$ corresponds to a selection of morphisms $f_\bullet \in \En{\X}(x_\bullet,y_\bullet)$ which happen to be atoms of their respective posets, whose sources and targets are all distinct, and which satisfy the global acyclicity assumption imposed by the partial order arising from $\prec_\mu$. 
\begin{definition}
\label{def:morsys}
A {\bf Morse system} $\Sigma$ on a cellular category $\bE$ is a collection 
\[
\{f_\bullet:x_\bullet \to y_\bullet\}
\] of morphisms subject to the following axioms.
\begin{enumerate}
\item {\bf Exhaustion:} if $f:x \to y$ lies in $\Sigma$, then it is the atom of $\bE(x,y)$ where $x \neq y$; no other morphism in $\Sigma$ has, as either source or target, any object $w \in \bE_0$ satisfying $\bE(x,w) \neq \varnothing \neq \bE(w,y)$. (In particular, neither $x$ nor $y$ is admissible as the source or target of any morphism in $\Sigma$ different from $f$).
\item {\bf Order:} the transitive closure of the binary relation $\flord$ defined on $\Sigma$ by
\[
(f_0: x_0 \to y_0) \flord (f_1:x_1 \to y_1) \text{ whenever } \bE(x_0,y_1) \text{ is nonempty }
\]
defines a partial order on $\Sigma$.
\item {\bf Lifting:} given $f_0 \neq f_1$ in $\Sigma$, each square of the form below splits via a horizontal morphism:
\[
\xymatrixcolsep{0.35in}
\xymatrixrowsep{0.25in}
\xymatrix{%
	& x_0 \ar@{->}[dr]^{g} \ar@{->}[dl]_{f_0}  & & &  & x_0 \ar@{->}[dr]^{g} \ar@{->}[dl]_{f_0} &   \\
	 y_0 \ar@{->}[dr]_{g'}  & \implies & x_1 \ar@{->}[dl]^{f_1} & \text{implies} & 	y_0 \ar@{->}[dr]_{g'} \ar@{->}[rr]_*+<0.8em>{\Longrightarrow}^*+<0.8em>{\Longrightarrow} &  & x_1 \ar@{->}[dl]^{f_1}   \\
 & y_1 & & &   & y_1 & 
	}
\]
\item {\bf Switching:} given $f_0 \neq f_1$ in $\Sigma$, if $\bE(x_0,x_1)$ and $\bE(y_0,y_1)$ are both nonempty and if their respective atoms $h$ and $\ell$ form the sides of the square below:
\[
\xymatrixcolsep{0.25in}
\xymatrixrowsep{0.25in}
\xymatrix{%
	& x_0 \ar@{->}[dr]^{h} \ar@{->}[dl]_{f_0}  \ar@{-->}[dd] &    \\
	 y_0 \ar@{->}[dr]_{\ell} & \twomor \hspace{2em} \twomorback &x_1 \ar@{->}[dl]^{f_1}     \\
 & y_1 & 
	}
\]
(for some dashed vertical morphism $v: x_0 \to y_1$), then $\bE(y_0,x_1)$ is nonempty, and its atom $q$ satisfies $f_0\com q\com f_1 \twomor v$.
\end{enumerate}
\end{definition}
The novelty here lies almost entirely in the {\bf lifting} and {\bf switching} axioms, and the former may be more explicitly stated as follows: given distinct $f_j:x_j \to y_j$ in $\Sigma$, if there exist $g: x_0 \to x_1$ and $g':y_0 \to y_1$ in $\bE$ satisfying $f_0\com g' \twomor g \com f_1$, then there also exists some $p: y_0 \to x_1$ simultaneously satisfying $f_0 \com p \twomor g$ and $g' \twomor p\com f_1$. 

\begin{proposition}
\label{prop:entcell}
Let $\X$ be a regular CW complex equipped with an acyclic partial matching $\mu$, and let $\Sigma = \{(x_\bullet > y_\bullet)\}$ denote the collection of entrance paths of $\X$ which correspond to the matchings $\mu(y_\bullet) = x_\bullet$. Then, $\Sigma$ is a Morse system on the entrance path category $\En{\X}$.
\end{proposition}
\begin{proof}
We only verify the {\bf lifting} and {\bf switching} axioms here, leaving the remaining verifications as simple exercises. In both arguments, we assume that $\mu(y_j) = x_j$ for $j \in \{0,1\}$ with $y_1$ and $x_1$ being proper faces of $y_0$ and $x_0$ respectively. Let $g = (x_0 > \cdots > x_1)$ and $g' = (y_0 > \cdots > y_1)$ be arbitrary entrance paths which satisfy 
\[
(x_0 > y_0) \com g' \twomor g \com (x_1 > y_1).
\]
By Definition \ref{def:ent2cat}, the path $g$ must contain $y_0$ as an intermediate cell. Thus, we have 
\[
g = (x_0 > \cdots > y_0 > \cdots > x_1).
\] The desired lift $p = (y_0 > \cdots > x_1)$ may now be obtained by discarding all the cells in $g$ which precede $y_0$. Turning now to the {\bf switching} axiom, note that any entrance path $v = (x_0 > \cdots > y_1)$ which satisfies
\[
(x_0 > y_0 > y_1) \twomor v \twomorback (x_0 > x_1 > y_1)
\]
must contain both $y_0$ and $x_1$ as intermediate cells. Furthermore, $y_0$ must precede $x_1$ in $v$ --- otherwise, writing $f_j:x_j \to y_j$ for the morphisms corresponding to the distinct entrance paths $(x_j > y_j)$, we violate the the acyclicity of $\mu$ since $f_0 \flord f_1 \flord f_0$. Therefore, 
\[
v = (x_0 > \cdots > y_0 > \cdots > x_1 > \cdots > y_1).
\] In particular, this forces $x_1$ to be a face of $y_0$ and the minimal entrance path $q = (y_0 > x_1)$ clearly satisfies
\[
f_0 \com q \com f_1 = (x_0 > y_0 > x_1 > y_1) \twomor v,
\]
as desired.
\end{proof}

\section{Localizations about Morse systems} \label{sec:lochom}

Every Morse system on a cellular category is directed in the sense of Definition \ref{def:loc2cat}, so we may localize about it. We will devote this section to proving the following result, which establishes one half of our main Theorem \ref{thm:main}.

\begin{theorem}
\label{thm:cellularloc}
If $\bE$ is a cellular category equipped with a Morse system $\Sigma$, then the (strict) localization p-functor $\bL_\Sigma:\bE \to \Loc_\Sigma\bE$ induces a homotopy-equivalence of classifying spaces.
\end{theorem}

Throughout this section, we adopt the notation above: $\bE$ will be a cellular category, and $\Sigma$ will be a Morse system defined on $\bE$. The underlying binary relation on $\Sigma$ which generates a partial order via the {\bf order} axiom of Definition \ref{def:morsys} is written $\flord$ as usual. We will always write constituent morphisms of $\Sigma$ as $f_\bullet:x_\bullet \to y_\bullet$ where the indexing $\bullet$'s are allowed to vary as needed. Let $\bL_\Sigma:\bE \to \Loc_\Sigma\bE$ be the canonical localization functor and fix an arbitrary object $z \in \bE_0$ --- we will show that the fiber $\bL_\Sigma \fiber z$ has a contractible classifying space, as the proof of Theorem \ref{thm:cellularloc} then follows from Theorem \ref{thm:quillenA}. While our fiber does not generally admit a homotopy maximal or minimal object in the sense of Definition \ref{def:homaxmin}, it may be mapped onto a contractible poset via a homotopy-equivalence. We begin by describing this target poset.

\subsection{Essential chains}

Recall from Definition \ref{def:fiber2cat} that an object of $\bL_\Sigma\fiber z$ is a pair $(w,\gamma)$ consisting of an object $w \in \bE_0$ and a morphism $\gamma:w \to z$ in $\Loc_\Sigma\bE$. And from Definition \ref{def:loc2cat}, it immediately follows (by removing all superfluous left-pointing identity maps) that any such $\gamma$ may be represented by a $\Sigma$-zigzag in $\bE$ which terminates at $z$:
\[
\xymatrixrowsep{0.4in}
\xymatrixcolsep{0.4in}
\xymatrix{
w \ar@{->}[r]^{g_0} & y_0 & x_0 \ar@{->}[l]_{f_0} \ar@{->}[r]^{g_1} & y_1 & \cdots \ar@{->}[l]_{f_1} &  x_k \ar@{->}[l]_{f_k} \ar@{->}[r]^{g_{k+1}} & z
}
\]
Here the left-pointing arrows $y_\bullet \gets x_\bullet$ must be chosen from $\Sigma$, and by the {\bf exhaustion} requirement of Definition \ref{def:morsys} none of the forward-pointing arrows (except possibly the extreme $g_0$ and $g_{k+1}$) are allowed to be identity maps. Moreover, the existence of $g_\bullet:x_\bullet \to y_{\bullet+1}$ forces our left-pointing morphisms to satisfy
\[
f_0 \flord f_1 \flord \cdots \flord f_k \text{ and } \bE(x_k,z) \neq \varnothing.
\]
by the {\bf order} axiom from Definition \ref{def:morsys}. Thus, each $\Sigma$-zigzag in $\bE$ furnishes a finite descending chain of elements in $\Sigma$ whose last element connects to $z$. The following definition puts such descending chains in their proper context.
\begin{definition}
Let $\Sigma_z$ be the poset of {\bf $\Sigma$-chains to $1_z$}, which has as its elements all strictly descending chains $\sigma = (f_0 \flord \cdots \flord f_k \flord 1_z)$ in $\Sigma$, where the last $\flord$ (in a mild abuse of notation) merely indicates that the source of $f_k:x_k \to y_k$ is required to satisfy $\bE(x_k,z) \neq \varnothing$. The partial order is given by inclusion: $\sigma \inc \tau$ whenever $\tau$ contains $\sigma$ as a (not necessarily contiguous) subsequence.
\end{definition}

A trivial application of Lemma \ref{lem:qinit} reveals that $\Sigma_z$ has a contractible classifying space: the empty chain $(1_z)$ clearly includes into all the others and hence serves as a minimal object. Our goal is to construct a strict p-functor $\bL_\Sigma\fiber z \to \Sigma_z$ which induces a homotopy-equivalence of classifying spaces, and hence yields the desired contractibility of $\bL_\Sigma\fiber z$. Unfortunately, the na\"{i}ve assignment of descending chains to $\Sigma$-zigzags described above is not well-defined, since distinct zigzags which represent the same morphism in $\Loc_\Sigma\bE$ might have different $\Sigma$-chains. Recall, for instance, the horizontal and vertical reductions of Remark \ref{rmk:reductions} which allowed us to identify
\[
\left(\cdots \stackrel{g\circ f}{\longrightarrow} y \stackrel{f}{\gets} x \stackrel{h}{\to} \cdots\right) \sim 
\left(\cdots \stackrel{g\circ h}{\longrightarrow} \cdots\right).
\]
It is clear that the zigzag on the left is assigned a $\Sigma$-chain containing $f$ while the zigzag on the right is not. The following definition is designed to rectify this defect.

\begin{definition}
\label{def:redundant}
Given a $\Sigma$-zigzag $\zeta$:
\[
\xymatrixrowsep{0.4in}
\xymatrixcolsep{0.4in}
\xymatrix{
w \ar@{->}[r]^{g_0} & y_0 & x_0 \ar@{->}[l]_{f_0} \ar@{->}[r]^{g_1} & y_1 & \cdots \ar@{->}[l]_{f_1} &  x_k \ar@{->}[l]_{f_k} \ar@{->}[r]^{g_{k+1}} & z
}
\]
and an index $j \in \{0,\ldots,k\}$, we call $f_j$
\begin{enumerate}
\item {\bf left-redundant} in $\zeta$ if $h\com f_j \twomor g_j$ for some $h:x_{j-1} \to x_j$,
\item {\bf right-redundant} in $\zeta$ if $f_j \com \ell \twomor g_{j+1}$ for some $\ell:y_j \to y_{j+1}$, and
\item {\bf essential} in $\zeta$ if it is neither left nor right-redundant.
\end{enumerate} 
(It is convenient here to adopt the convention $x_{-1} = w$ and $y_{k+1} = z$ in order to avoid separately mentioning these extreme cases). Without loss of generality, the $h$ and $\ell$ above may be chosen to be the atoms of $\bE(x_{j-1},x_j)$ and $\bE(y_{j-1},y_j)$ respectively.
\end{definition}

Here we illustrate a segment of $\zeta$ containing an $f_j$ which is both left-redundant (due to the triangle which precedes it) and right-redundant (due to the triangle which comes afterwards):
\[\xymatrixcolsep{0.3in}
\xymatrixrowsep{0.2in}
\xymatrix{%
 & x_j \ar@{->}[dr]^{f_j} \ar@{=>}[d] & & &   \\
 x_{j-1} \ar@{-->}[ur]^{h} \ar@{->}[rr]_{g_j} & & y_j  & x_j \ar@{->}[l]_{f_j} \ar@{->}[dr]_{f_j} \ar@{->}[rr]^{g_{j+1}} & & y_{j+1} \\
& & & & y_j \ar@{-->}[ur]_{\ell} \ar@{=>}[u] & &
	}
\]
\begin{remark}
\label{rmk:redcont}
Note that if $f_j$ is redundant in $\zeta$, then we must have $f_{j-1} \flord f_{j+1}$ since the poset $\bE(x_{j-1},y_{j+1})$ is nonempty --- it contains either $h \com g_{j+1}$ or $g_j \com \ell$ (or both) depending on whether $f_j$ is left or right-redundant (or both).  
\end{remark}
In light of the preceding remark, the collection of essential $f_\bullet$'s in $\zeta$ forms an element of $\Sigma_z$, which we denote by $\sigma_\zeta$ and call the {\bf essential chain} of $\zeta$. Note that $\sigma_\zeta = (1_z)$ whenever $\zeta$ contains no essential $f_\bullet$'s. 

\subsection{Well-definedness of essential chains}

The primary reason for introducing essential chains is the following result.
\begin{proposition}
\label{prop:esswelldef}
If $\zeta$ and $\zeta'$ are $\Sigma$-zigzags which represent the same morphism $\gamma:w \to z$ in $\Loc_\Sigma\bE$, then they have the same essential chain, i.e., $\sigma_\zeta = \sigma_{\zeta'}$. 
\end{proposition}
\begin{proof}
It suffices to establish that the essential chain of a $\Sigma$-zigzag remains unchanged when vertical reductions (in the sense of Remark \ref{rmk:reductions}) are performed, since horizontal reductions only remove identity maps and hence leave the $\Sigma$-chain of a zigzag invariant. So without loss of generality, we may assume that $\zeta$ and $\zeta'$ form the top and bottom rows of the diagram below:
\[
\xymatrixcolsep{0.4in}
\xymatrixrowsep{0.4in}
\xymatrix{
w \ar@{->}[r]^{g_0} \ar@{=}[d]_{} &  z_0  \ar@{->}[d]^{} & w_0 \ar@{->}[l]_{} \ar@{->}[r]^{g_1} \ar@{->}[d]^{} & \cdots & w_{k} \ar@{->}[r]^{g_{k+1}} \ar@{->}[l]_{} \ar@{->}[d]_{} & z \ar@{=}[d] \\
w \ar@{->}[r]_{g'_0} \ar@{=}[ur]^{}	&  z'_0 	& w'_0 \ar@{->}[l]^{}  	\ar@{->}[r]_{g'_1} \ar@{=}[ul]_{} \ar@{=}[ur]_{} & \cdots & w'_{k} \ar@{->}[l]^{} \ar@{->}[r]_{g'_{k+1}} \ar@{=}[ul]^{} \ar@{=}[ur]^{}& z
}
\]
where all leftward and downward-pointing arrows denote morphisms in $\Sigma^+$ (i.e., either elements of $\Sigma$ or identities). We may further simplify the diagram by insisting that --- performing horizontal reductions to both rows if necessary --- each column (except possibly the first and last ones) has at most one horizontal identity map. Now by the {\bf exhaustion} requirement of Definition \ref{def:morsys}, there are only three possible configurations of the backward-pointing columns:
\[
\xymatrixcolsep{0.4in}
\xymatrixrowsep{0.4in}
\xymatrix{
x_\bullet \ar@{->}[d]_{f_\bullet} & x_\bullet \ar@{=}[l] \ar@{=}[d] & 
y_\bullet \ar@{=}[d] & x_\bullet \ar@{->}[l]_{f_\bullet} \ar@{->}[d]^{f_\bullet} &
y_\bullet \ar@{=}[d] & x_\bullet \ar@{->}[l]_{f_\bullet} \ar@{=}[d] \\
y_\bullet  & x_\bullet \ar@{->}[l]^{f_\bullet} \ar@{=}[ul] &
y_\bullet  & y_\bullet \ar@{=}[l] \ar@{=}[ul] &
y_\bullet  & x_\bullet \ar@{->}[l]^{f_\bullet} \ar@{=}[ul]
}
\]
The desired result now follows from a brief examination of these configurations. As an example, consider the first configuration above where $f_\bullet$ appears in $\zeta'$ but not in $\zeta$. We claim that this $f_\bullet$ must be left-redundant in $\zeta'$. There are only two possibilities for what could appear on the left:
\[
\xymatrixcolsep{0.4in}
\xymatrixrowsep{0.4in}
\xymatrix{
x_{\bullet-1} \ar@{=}[d]_{} \ar@{->}[r]^{g_{\bullet}} & x_\bullet \ar@{->}[d]^{f_\bullet} & 
x_{\bullet-1} \ar@{->}[r]^{g_{\bullet}} \ar@{->}[d]_{f_{\bullet-1}} & x_\bullet \ar@{->}[d]^{f_\bullet} \\
x_{\bullet-1} \ar@{=}[ur] \ar@{->}[r]_{g'_{\bullet}} & y_\bullet  &
y_{\bullet-1} \ar@{->}[r]_{g'_{\bullet}} \ar@{=}[ur] & y_\bullet  
}
\]
In the first case, we have $g_\bullet \com f_\bullet = g'_\bullet$ and so it is immediate from Definition \ref{def:redundant} that $f_\bullet$ is left-redundant in $\zeta'$. In the second case, one applies the {\bf lifting} axiom from Definition \ref{def:morsys} and concludes the existence of some $p:y_{\bullet-1} \to x_{\bullet}$ satisfying $p \com f_\bullet \twomor g'_\bullet$. Again, this forces $f_\bullet$ to be left-redundant in $\zeta'$. Similar arguments on the remaining configurations yield the desired conclusion.
\end{proof}

Armed with the preceding proposition, we may unambiguously associate to each object $(w,\gamma)$ in the fiber $\bL_\Sigma\fiber z$ the element $\sigma_\gamma$ of $\Sigma_z$ which equals the essential chain of {\em any} zigzag representing the morphism $\gamma \in \Loc_\Sigma\bE(w,z)$.
\begin{definition}
Let $\bN:(\bL_\Sigma \fiber z)_0 \to (\Sigma_z)_0$ be the function which sends each object of $\bL_\Sigma \fiber z$ to its essential chain in $\Sigma_z$. That is, $\bN(w,\gamma) = \sigma_\gamma$. 
\end{definition}

\subsection{Contractibility of the fibers}

In this subsection we establish two facts --- first, that the function $\bN$ defined above actually prescribes a functor from the fiber $\bL_\Sigma\fiber z$ to the contractible poset $\Sigma_z$ of descending chains; and second, that our functor $\bN: \bL_\Sigma \fiber z \to \Sigma_z$ induces a homotopy-equivalence on classifying spaces. The following result accomplishes the first task.

\begin{proposition}
\label{prop:nfunc}
Let $(w,\gamma)$ and $(w',\gamma')$ be objects of the fiber $\bL_\Sigma\fiber z$. If there exists a morphism $h:(w,\gamma) \to (w',\gamma')$, then there is an inclusion $\bN(w,\gamma) \inc \bN(w',\gamma')$ of the corresponding essential chains in $\Sigma_z$. 
\end{proposition}
\begin{proof}
The argument proceeds very similarly to the one employed in the proof of Proposition \ref{prop:esswelldef}. Given a morphism $h:(w,\gamma) \to (w',\gamma')$ in $\bL_\Sigma\fiber z$, there must exist (suitably reduced) $\Sigma$-zigzags representing $\gamma$ and $\gamma'$ respectively which form the top and bottom rows of the following diagram:
\[
\xymatrixcolsep{0.4in}
\xymatrixrowsep{0.4in}
\xymatrix{
w \ar@{->}[r]^{g_0} \ar@{->}[d]_{h} &  z_0  \ar@{->}[d]^{} & w_0 \ar@{->}[l]_{} \ar@{->}[r]^{g_1} \ar@{->}[d]^{} & \cdots & w_{k} \ar@{->}[r]^{g_{k+1}} \ar@{->}[l]_{} \ar@{->}[d]_{} & z \ar@{=}[d] \\
w' \ar@{->}[r]_{g'_0} \ar@{=>}[ur]^{}	&  z'_0 	& w'_0 \ar@{->}[l]^{}  	\ar@{->}[r]_{g'_1} \ar@{=>}[ul]_{} \ar@{=>}[ur]_{} & \cdots & w'_{k} \ar@{->}[l]^{} \ar@{->}[r]_{g'_{k+1}} \ar@{=>}[ul]^{} \ar@{=>}[ur]^{}& z
}
\]
Here all leftward and downward pointing arrows (except $h$) are chosen from $\Sigma^+$ as usual, and each column (except possibly the first and last ones) contains at most one horizontal identity morphism. The backward-pointing columns may once again only assume one of three possible configurations:
\[
\xymatrixcolsep{0.4in}
\xymatrixrowsep{0.4in}
\xymatrix{
x_\bullet \ar@{->}[d]_{f_\bullet} & x_\bullet \ar@{=}[l] \ar@{=}[d] & 
y_\bullet \ar@{=}[d] & x_\bullet \ar@{->}[l]_{f_\bullet} \ar@{->}[d]^{f_\bullet} &
y_\bullet \ar@{=}[d] & x_\bullet \ar@{->}[l]_{f_\bullet} \ar@{=}[d] \\
y_\bullet  & x_\bullet \ar@{->}[l]^{f_\bullet} \ar@{=}[ul] &
y_\bullet  & y_\bullet \ar@{=}[l] \ar@{=}[ul] &
y_\bullet  & x_\bullet \ar@{->}[l]^{f_\bullet} \ar@{=}[ul]
}
\]
We now argue by contrapositive and show that if $f_\bullet$ is redundant in (or missing altogether from) $\gamma'$, then it must also be redundant in $\gamma$. Note that the first configuration above poses no threat, since $f_\bullet$ appears in $\gamma'$ but not in $\gamma$. We analyze the second configuration here and leave the third to the reader. The only possibilities for right-pointing columns which could appear after the second configuration are:
\[
\xymatrixcolsep{0.4in}
\xymatrixrowsep{0.4in}
\xymatrix{
x_{\bullet} \ar@{->}[d]_{f_\bullet} \ar@{->}[r]^{g_{\bullet+1}} & y_{\bullet+1} \ar@{=}[d] & 
x_{\bullet} \ar@{->}[r]^{g_{\bullet+1}} \ar@{->}[d]_{f_{\bullet}} & x_{\bullet+1} \ar@{->}[d]^{f_{\bullet+1}} \\
y_{\bullet} \ar@{=>}[ur] \ar@{->}[r]_{g'_{\bullet+1}} & y_{\bullet+1}  &
y_{\bullet} \ar@{->}[r]_{g'_{\bullet+1}} \ar@{=>}[ur] & y_{\bullet+1} 
}
\]
In the first case, $f_\bullet$ is right-redundant in $\gamma$ by definition, whereas in the second case it is right-redundant via an application of the {\bf lifting} axiom. In particular, there is some $p:y_\bullet \to x_{\bullet + 1}$ with $f_\bullet \com p \twomor g_{\bullet+1}$ (as desired).
\end{proof}

The preceding result guarantees that $\bN:\bL_\Sigma\fiber z \to \Sigma_z$ is a functor, so we now examine its fibers to show that it induces a homotopy-equivalence. We will proceed by induction over the poset $\Sigma_z$, using the following proposition as a base (recall that $(1_z)$ is the minimal chain in $\Sigma_z$).

\begin{proposition}
\label{prop:base}
The fiber $\bN\fiber (1_z)$ has a contractible classifying space.
\end{proposition}
\begin{proof}
We will show that the object $(z,1_z)$ is homotopy-maximal in $\bN\fiber (1_z)$ and obtain the desired contractibility by Lemma \ref{lem:qinit}. Meditating on the following diagram of $\Loc_\Sigma\bE$:
\[
\xymatrixcolsep{0.3in}
\xymatrixrowsep{0.4in}
\xymatrix{
w \ar@{..>}[rr]_*+<1.4em>{{\Longleftarrow}}^{} \ar@{->}[dr]_{\gamma} & & z \ar@{->}[dl]^{1_z} \\
 & z & 
}
\]
we must show that the subposet $\bQ_\gamma \subset \bE(w,z)$ defined by
\[
\bQ_\gamma = \{g \in \bE(w,z) \mid g \twomor \gamma \text{ in }\Loc_\Sigma\bE\}
\]
has a minimal element\footnote{Minimal objects in $\bE(w,z)$ yield maximal objects in $\bL_\Sigma\fiber z$, thanks to the contravariance of partial orders mentioned in Definition \ref{def:fiber2cat}.} (here we have identified $g$ with its image $\bL_\Sigma g$). It suffices to show that $\bQ_\gamma$ is nonempty, since the atom of $\bE(w,z)$ would then constitute the desired minimal element. To this end, assume that the following $\Sigma$-zigzag represents the morphism $\gamma$
\[
\xymatrixrowsep{0.4in}
\xymatrixcolsep{0.4in}
\xymatrix{
w \ar@{->}[r]^{g_0} & y_0 & x_0 \ar@{->}[l]_{f_0} \ar@{->}[r]^{g_1} & y_1 & \cdots \ar@{->}[l]_{f_1} &  x_k \ar@{->}[l]_{f_k} \ar@{->}[r]^{g_{k+1}} & z
}
\]
Since $\bN(w,\gamma) = (1_z)$, all the $f_\bullet$'s appearing above are (left, right or multiply) redundant. By Definition \ref{def:redundant}, this zigzag therefore forms the solid, non-horizontal arrows in some diagram of $\bE$ which resembles the following one:
\[
\xymatrixrowsep{0.5in}
\xymatrixcolsep{0.5in}
\xymatrix{
w \ar@{..>}[r]^{h_0} \ar@{->}[d]_{} & x_0 \ar@{..>}[r]^{h_1} \ar@{->}[dl]_*+<0.5em>{\Longleftarrow}^*+<0.5em>{\Longrightarrow} \ar@{->}[d] & \cdots 
\ar@{..>}[r]^{h_{k-1}} \ar@{->}[dl]_*+<0.5em>{\Longleftarrow}^*+<0.5em>{\Longrightarrow} & x_{k-1} \ar@{->}[dl]_*+<0.5em>{\Longleftarrow}^*+<0.5em>{\Longrightarrow}  
\ar@{->}[d] \ar@{..>}[r]^{h_k} & x_k \ar@{->}[dl]_*+<0.5em>{\Longleftarrow}^*+<0.5em>{\Longrightarrow}  \ar@{->}[d]\\
y_0 \ar@{..>}[r]_{\ell_0} & y_1 \ar@{..>}[r]_{\ell_1} & \cdots \ar@{..>}[r]_{\ell_{k-1}} & y_k \ar@{..>}[r]_{\ell_{k}} & z
}
\]
Here each column contains either an $h_\bullet$ above or an $\ell_\bullet$ below (and possibly both), depending on whether the diagonal $f_\bullet$ is left or right-redundant in $\gamma$ (or both). Moreover, we can assume that all $h_\bullet$ and $\ell_\bullet$'s in sight, whenever they exist, are atomic in their respective posets. All directed paths  from $w$ to $z$ in our diagram prescribe compositions in $\bE$ --- such as the extremal $h_0 \com \cdots \com h_k \com g_{k+1}$ and $g_0 \com \ell_0 \com \cdots \com \ell_k$ --- and yield elements of $\bQ_\gamma$.  The only admissible configuration of $h_\bullet$'s and $\ell_\bullet$'s where a direct path is not immediately available arises when some $f_{\bullet-1}$ is right (but not left) redundant and the next $f_{\bullet}$ is left (but not right) redundant:
\[
\xymatrixrowsep{0.5in}
\xymatrixcolsep{0.5in}
\xymatrix{
x_{\bullet-2} \ar@{->}[d]_{g_{\bullet-1}} & x_{\bullet-1} \ar@{->}[d]_{g_{\bullet}}  \ar@{->}[dl]^*+<0.5em>{\Longrightarrow}_{f_{\bullet-1}} \ar@{..>}[r]^{h_{\bullet}} & x_{\bullet} \ar@{->}[d]^{g_{\bullet+1}} \ar@{->}[dl]_*+<0.5em>{\Longleftarrow}^{f_\bullet} \\
y_{\bullet-1}\ar@{..>}[r]_{\ell_{\bullet-1}} & y_{\bullet} & y_{\bullet + 1}
}
\]
But in this case, the {\bf switching} axiom of Definition \ref{def:morsys} applies to the parallelogram above and guarantees the existence of the atom $q_\bullet:y_{\bullet-1} \to x_{\bullet}$ in $\bE$ which furnishes a connection from $x_{\bullet-2}$ to $x_{\bullet}$, so in general one can connect $w$ to $z$ via a directed path even when confronting the unfortunate configuration described above. Moreover, since $q_\bullet$ is required to satisfy $f_{\bullet-1} \com q_\bullet \com f_{\bullet} \twomor g_{\bullet}$, one can easily check that the directed path so obtained also creates an element of $\bQ_\gamma$, as desired.
\end{proof}

This next proposition forms the final, inductive step when analyzing the fibers of our functor $\bN:\bL_\Sigma\fiber z \to \Sigma_z$. Since the target category $\Sigma_z$ is a poset, the fiber $\bN \fiber \sigma$ over a descending chain $\sigma = (f_0 \flord \cdots \flord f_k \flord 1_z)$ is precisely the full subcategory of $\bL_\Sigma\fiber z$ generated by all objects $(w,\gamma)$ satisfying $\bN(w,\gamma) \inc \sigma$. Thus, {\em all fibers of $\bN$ encountered henceforth will be treated as full subcategories of $\bL_\Sigma\fiber z$}, keeping in mind that if $\tau \inc \sigma$ holds for two descending chains in $\Sigma_z$ then $\bN\fiber\tau$ includes into $\bN\fiber\sigma$ as a full subcategory.

\begin{proposition}
\label{prop:indstep}
Let $\sigma = (f_0 \flord \cdots \flord f_k \flord 1_z)$ be a non-trivial descending chain in $\Sigma_z$, and assume that the union $\bA = \bigcup_\tau (\bN \fiber \tau)$ of $\bN$'s fibers over $\{\tau \in \Sigma_z \mid \sigma \neq \tau \inc \sigma\}$ is a contractible subcategory of $\bL_\Sigma\fiber z$. Then, the fiber $\bN \fiber \sigma$ is also contractible.
 \end{proposition}
\begin{proof}
Consider the object $(y_0,\gamma_\sigma)$ in $\bL_\Sigma\fiber z$ given by the zigzag
\[
\xymatrixrowsep{0.4in}
\xymatrixcolsep{0.4in}
\xymatrix{
y_0 \ar@{->}[r]^{a_0=1} & y_0 & x_0 \ar@{->}[l]_{f_0} \ar@{->}[r]^{a_1} & y_1 & \cdots \ar@{->}[l]_{f_1} &  x_k \ar@{->}[l]_{f_k} \ar@{->}[r]^{a_{k+1}} & z
}
\]
where the $a_\bullet:x_{\bullet-1} \to y_\bullet$ are atoms in $\bE(x_{\bullet-1},y_{\bullet})$. Define the full subcategory $\bB$ of $\bL_\Sigma\fiber z$ generated by the following subset of objects: 
\[
\bB_0 = \big{\{}(w,\gamma) \mid \text{there exists some } h:(w,\gamma) \to (y_0,\gamma_\sigma) \text{ in }\bL_\Sigma\fiber z\big{\}}.
\]
It is easy to check that $\bB$ is contractible (since $(y_0,\gamma_\sigma)$ is homotopy-maximal) and that every object of $\bN\fiber \sigma$ not in $\bA$ belongs to $\bB$ (since any $(w,\gamma)$ with essential chain equaling $\sigma$ admits a morphism to $(y_0,\gamma_\sigma)$ by minimality of the $a_\bullet$'s). We therefore have a decomposition
\[
\bN\fiber\sigma = \bA \cup \bB,
\]
where $\bA$ and $\bB$ are contractible full subcategories of $\bN \fiber \sigma$. It now suffices to show that the intersection $\bA \cap \bB$ is contractible\footnote{Consult, for instance, \cite[Ex. 0.23]{hatcher}.}. Any $(w,\gamma)$ in $\bB_0$ must be expressible as the top row of the following diagram, which illustrates a morphism $h$ to $(y_0,\gamma_\sigma)$:
\[
\xymatrixcolsep{0.4in}
\xymatrixrowsep{0.4in}
\xymatrix{
w \ar@{->}[r]^{g_0} \ar@{->}[d]_{h} &  y_0  \ar@{=}[d]^{} & x_0 \ar@{->}[l]_{f_0} \ar@{->}[r]^{g_1} \ar@{=}[d]^{} & \cdots & x_{k} \ar@{->}[r]^{g_{k+1}} \ar@{->}[l]_{f_k} \ar@{=}[d]_{} & z \ar@{=}[d] \\
y_0 \ar@{=}[r]_{} \ar@{=>}[ur]^{}	&  y_0 	& x_0 \ar@{->}[l]^{f_0}  	\ar@{->}[r]_{a_1} \ar@{=}[ul]_{} \ar@{=>}[ur]_{} & \cdots & x_{k} \ar@{->}[l]^{} \ar@{->}[r]_{a_{k+1}} \ar@{=}[ul]^{} \ar@{=>}[ur]^{}& z
}
\]
But if $(w,\gamma)$ is also in $\bA_0$, then it has an essential chain strictly smaller than $\sigma$, so at least one of the $f_j$ must be left or right-redundant in the top row. Thus, $\bA \cap \bB$ admits a cover by $2(k+1)$ full subcategories $\bU^\pm_j$ (where $0 \leq j \leq k$) that may be defined via conditions on the $g_\bullet$'s which force the $f_\bullet$'s to be left or right-redundant. So for each $0 \leq j \leq k$,  
\begin{itemize}
\item $\bU^+_j$ contains zigzags satisfying $h_j\com f_j \twomor g_j$ where $h_j$ is atomic in $\bE(x_{j-1},x_j)$, and 
\item$\bU^-_j$ contains zigzags satisfying $f_j\com\ell_{j+1} \twomor g_{j+1}$ where $\ell_{j+1}$ is atomic in $\bE(y_j,y_{j+1})$. 
\end{itemize}
We allow for the possibility that $\bU^+_\bullet$ or $\bU_\bullet^-$ might be empty if no such $h_\bullet$ or $\ell_\bullet$ exist. If $\bU^+_j$ and $\bU^-_{j+1}$ are both nonempty, then recall that the {\bf switching} axiom of Definition \ref{def:morsys} guarantees the presence of an atom $q_j:y_j \to x_{j+1}$.

Next, we show that all non-empty intersections of the $\bU^\pm_j$ are contractible. Given a subset $\mathbb{U}$ of $\{\bU^\pm_j\}$ whose constituent categories intersect non-trivially, one can readily construct the homotopy-maximal element residing in that intersection by choosing all the appropriate forward-pointing $g_j:x_{j-1} \to y_j$ in the zigzag representative above. In particular, setting
\begin{align*}
g_j = \begin{cases}
							a_j & \text{ if neither } \bU^+_j \text{ nor } \bU^-_{j+1} \text{ is in } \mathbb{U}, \\
							h_j \com f_j & \text{ if } \bU^+_j \text{ is in } \mathbb{U} \text{ but } \bU^-_{j+1} \text{ is not}, \\
							f_j \com \ell_{j+1} & \text{ if } \bU^+_j \text{ is not in } \mathbb{U} \text{ but } \bU^-_{j+1} \text{ is}, \\
							f_j \com q_j \com f_{j+1} & \text{ if both } \bU^+_j \text{ and } \bU^-_{j+1} \text{ are in } \mathbb{U},
					\end{cases}
\end{align*}
creates a zigzag representing the desired homotopy-maximal element in the intersection of all subcategories contained in $\mathbb{U}$. It is easy to check that $\bU^+_0$ is always non-empty (set $w = x_0$ and $g_0 = f_0 = h$ in the zigzag above) and that it always intersects any nonempty subcollection of categories in $\mathbb{U}$ (since one can simply modify $g_0$ independently of the other $g_j$'s). We now appeal to a suitable version of the {\bf nerve theorem}\footnote{The {\em nerve} of a locally finite cover of some CW complex by subcomplexes is that abstract simplicial complex whose $d$-simplices correspond to $(d+1)$-fold nonempty intersections (i.e., what one obtains by applying Definition \ref{def:geonerve} to the intersection lattice). The nerve theorem states that if all such intersections are contractible, then the underlying CW complex is homotopy-equivalent to the nerve.} (see \cite[Thm 15.21 \& Rmk 15.22]{Kozlov08}): the intersection $\bA \cap \bB$ is covered by at most $2(k+1)$ full subcategories, all of whose nonempty intersections are contractible. Moreover, the nerve of this cover contains a distinguished vertex $\bU^+_0$ lying in every maximal simplex. Thus, the nerve of the cover contracts to a single vertex and $\bA \cap \bB$ is consequently contractible. 
\end{proof}

By Propositions \ref{prop:base} and \ref{prop:indstep}, the functor $\bN:\bL_\Sigma \fiber z\to \Sigma_z$ has a contractible fiber $\bN \fiber \sigma$ over each $\sigma \in \Sigma_z$. By Theorem \ref{thm:quillenA}, $\bN$ induces a homotopy-equivalence between classifying spaces $|\Delta(\bL_\Sigma\fiber z)|$ and $|\Delta(\Sigma_z)|$. Since the poset $\Sigma_z$ has a minimal element $(1_z)$, it is contractible, and hence so is the fiber $\bL_\Sigma\fiber z$. Finally, since $z \in \bE_0$ was chosen arbitrarily, the functor $\bL_\Sigma:\bE \to \Loc_\Sigma\bE$ also induces a homotopy-equivalence of classifying spaces by Theorem \ref{thm:quillenA}, thus concluding the proof of Theorem \ref{thm:cellularloc}.  

\section{The discrete flow category} \label{sec:flocat}

As before, we let $\bE$ be a cellular category equipped with a Morse system $\Sigma$ whose binary relation is $\flord$ and ask the reader to consult Definitions \ref{def:cell2cat} and \ref{def:morsys} for details. 

\subsection{The subcategory of critical objects}

An object $m \in \bE_0$ is called {\bf $\Sigma$-critical} if for each $f:x \to y$ in $\Sigma$ at least one of the posets $\bE(x,m)$ and $\bE(m,y)$ is empty. It is easily checked (as those familiar with discrete Morse theory might expect) that both the source $x$ and target $y$ of every $f:x \to y$ in $\Sigma$ are not $\Sigma$-critical. 

For every non-critical $w$ in $\bE_0$ there is in fact a {\em unique} $f:x \to y$ in $\Sigma$ whose source and target satisfy $\bE(x,w) \neq \varnothing \neq \bE(w,y)$. To see why this is the case, note that if two different $f$ and $f'$ satisfied this condition, then we would obtain $f \flord f' \flord f$ and violate the {\bf order} axiom of Definition \ref{def:morsys}. We write $S(f)$ to indicate the set of all objects which are rendered non-critical by a given $f \in \Sigma$, and call this set the {\bf span} of $f$. Clearly, both $x$ and $y$ lie in the span of $f:x \to y$. In the special case where the Morse system $\Sigma$ is induced by an acyclic partial matching, we have  $S(f) = \{x,y\}$, since by the {\bf incidence} requirement of Definition \ref{def:partmatch} there are no intermediate cells $z$ satisfying $x > z > y$.

The following category is built around the critical objects and it plays the role of a Morse complex in our setting.
\begin{definition}
\label{def:flo}
The {\bf discrete flow category} $\Flo_\Sigma\bE$ of $\Sigma$ is the full subcategory of the localization $\Loc_\Sigma\bE$ generated by the set of $\Sigma$-critical objects.
\end{definition}

With an eye towards proving the half of Theorem \ref{thm:main} not addressed by Theorem \ref{thm:cellularloc}, we would like to establish that the inclusion $\bJ: \Flo_\Sigma\bE \to \Loc_\Sigma\bE$, a strict p-functor, induces a homotopy-equivalence of classifying spaces. But the desired equivalence does not hold for arbitrary cellular categories and Morse systems: perhaps the simplest illustrations of this failure are given below.
\begin{example}
\label{ex:pathology}
Here are two (and a half) instances where the localized entrance path category has a different homotopy type from the flow category.
\begin{enumerate}
\item Consider the poset $\bbP$ with four objects $x,y,m$ and $m'$ where the only non-trivial order relations are $x > y$ along with $x > m$ and $x > m'$. One can easily verify that ${\bbP}$ is cellular (when treated as a p-category with trivial poset structures on its morphism-sets) and has a contractible classifying space, since $x$ is a maximal element. Impose the singleton Morse system $\{(x > y)\}$ on $\bbP$, and note that the associated discrete flow category is not contractible --- it has the homotopy type of two points since it consists of $m$ and $m'$ with no morphisms between them.
\item Our second example involves the acyclic partial matching depicted on the semi-infinite collection of cubes below:
\begin{center}
\includegraphics[scale=0.2]{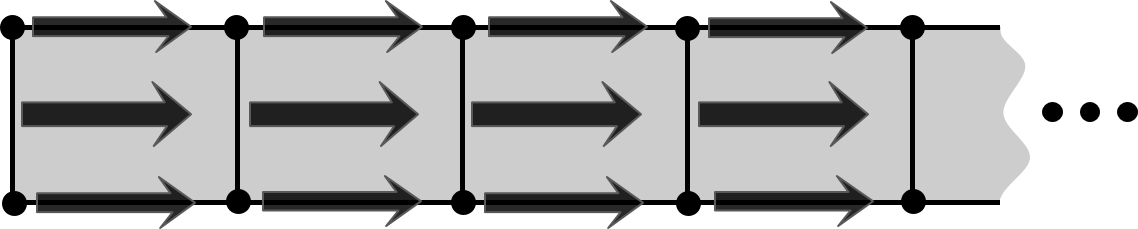}
\label{fig:infmatch}
\end{center}
Although this cube complex is contractible, the flow category associated to the overlaid acyclic matching is empty since there are no critical cells whatsoever. If the vertices are matched with edges to their left rather than the right, like so:
\begin{center}
\includegraphics[scale=0.2]{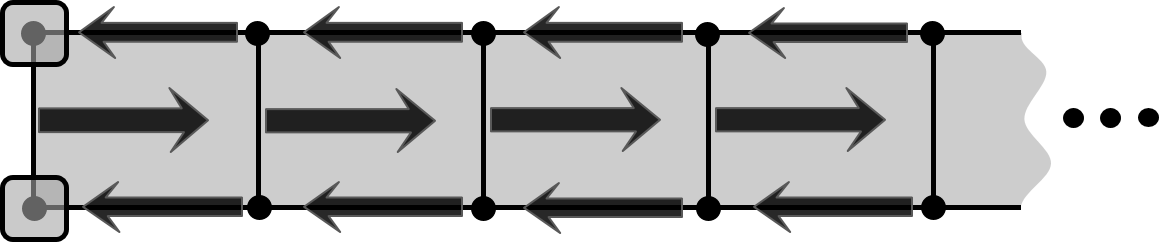}
\label{fig:infmatch2}
\end{center}
then once again we have a flow category with the incorrect homotopy type --- on this occasion, there are two objects (generated by the two boxed critical vertices on the left edge of the first cube) and no morphisms between them.
\end{enumerate}
\end{example}

A more immediate difficulty encountered here is that studying the fibers $\bJ\fiber w$ of $\bJ$ over objects $w$ of $\Loc_\Sigma\bE$ is likely to be fruitless: a brief examination of the acyclic partial matching from Figure \ref{fig:match} confirms the existence of cells (such as the 2-simplex $wxy$) which do not admit any morphisms from critical cells in the localized entrance path category, and hence have empty over-fibers. On the other hand, it appears as though each cell does admit morphisms {\em to} some critical cell (as opposed to the second example above). We therefore explore fibers dual to the ones described in Definition \ref{def:fiber2cat}.

\begin{definition}
\label{def:underfib}
Let $w$ be any object of $\bE$ (or equivalently, of $\Loc_\Sigma\bE$) and recall the inclusion functor $\bJ:\Flo_\Sigma\bE \to \Loc_\Sigma\bE$. The {\bf fiber of $\bJ$ under $w$} is the p-category $w \fiber \bJ$ whose
\begin{enumerate}
\item objects are pairs $(m,\gamma)$ where $m$ is $\Sigma$-critical and $\gamma \in \Loc_\Sigma\bE(w,m)$,
\item morphisms from $(m,\gamma)$ to $(m',\gamma')$ are given by $\rho \in \Loc_\Sigma\bE(m,m')$ satisfying
\[
\xymatrixcolsep{0.25in}
\xymatrixrowsep{0.35in}
\xymatrix{
 & w \ar@{->}[dl]_{\gamma} \ar@{->}[dr]^{\gamma'} & \\ 
m \ar@{->}[rr]^*+<1.2em>{{\Longleftarrow}}_{\rho} & & m' \\
}
\]
\item and the poset structure is inherited from $\Loc_\Sigma\bE(m,m')$. That is, $\rho \twomor \rho'$ holds in the fiber iff it holds in $\Loc_\Sigma\bE$.
\end{enumerate}
\end{definition}
The rules for composing morphisms are also inherited from $\Loc_\Sigma\bE$ in a straightforward manner dual to those from Definition \ref{def:fiber2cat}. The conclusion of Quillen's result (see Theorem \ref{thm:quillenA} above) also holds when all under-fibers (rather than over-fibers) are contractible, so it suffices to impose conditions which guarantee that $w \fiber \bJ$ is contractible for each object $w$ of $\bE$.  The next result shows that there is no difficulty at least when dealing with fibers under critical objects. 
\begin{proposition}
\label{prop:critfiber}
The fiber $m \fiber \bJ$ of $\bJ$ under $m$ is contractible whenever $m$ is $\Sigma$-critical.
\end{proposition} 
\begin{proof}
Note that the object $(m,1_m)$ in $m \fiber \bJ$ is homotopy-minimal, since for any other object $(m',\gamma)$ of $m \fiber \bJ$ the collection of admissible $\rho:m \to m'$ from $\Loc_\Sigma\bE$ in the diagram below
\[
\xymatrixcolsep{0.25in}
\xymatrixrowsep{0.35in}
\xymatrix{
 & m \ar@{->}[dl]_{1_m} \ar@{->}[dr]^{\gamma} & \\ 
m \ar@{->}[rr]^*+<1.2em>{{\Longleftarrow}}_{\rho} & & m' \\
}
\]
contains $\gamma$ as a minimal element. An appeal to Lemma \ref{lem:qinit} concludes the argument.
\end{proof}

\subsection{Mild Morse systems}

Our next goal is to impose additional hypotheses on $\Sigma$ (beyond the requirements of Definition \ref{def:morsys}) which preclude the undesirable phenomena described in Example \ref{ex:pathology} and hence guarantee that fibers of $\bJ$ under non-critical objects are also contractible.
\begin{definition}
\label{def:mild}
The Morse system $\Sigma$ on the cellular category $\bE$ is called {\bf mild} if for each $f:x \to y$ in $\Sigma$ the following two conditions hold. 
\begin{enumerate}
\item Every strictly descending $\Sigma$-chain of the form 
\[
f \flord f_0 \flord f_1 \flord f_2 \flord \cdots
\]
eventually terminates at some locally minimal $f_k$. Here, local minimality means that if $f' \in \Sigma$ satisfies $f_k \flord f'$, then $f' = f_k$.
\item The full subcategory $\bE|f$ of $\bE$ generated by all objects $z \in \bE_0$ simultaneously satisfying $\bE(x,z) \neq \varnothing$ and $\bE(z,y) = \varnothing$ is:
\begin{itemize} 
\item (finite): it has finitely many objects $z$, 
\item (loopfree): if $\bE(z,z') \neq \varnothing$ then $\bE(z',z) = \varnothing$ for $z \neq z'$ in $(\bE|f)_0$, and
\item (contractible): the classifying space $|\Delta(\bE|f)|$ is contractible.
\end{itemize}
\end{enumerate}
\end{definition}
The partial matching from Example \ref{ex:pathology}(1) violates the second mildness condition above --- writing $f$ for the matched entrance path $(x > y)$ in the poset $\bbP$, note that the category ${\bbP}|f$ is not contractible since it consists of two disconnected objects $m$ and $m'$. Similarly, both acyclic partial matchings depicted in Example \ref{ex:pathology}(2) violate the first mildness condition due to the infinite descending chain which consists of all the 2-dimensional cubes paired with their left edges. The following result confirms that mildness is not an unreasonable constraint in the familiar case of acyclic partial matchings on finite regular CW complexes.

\begin{proposition}
\label{prop:finmild}
Let $\X$ be a regular CW complex equipped with an acyclic partial matching $\mu$. If $\X$ is finite, then, the Morse system $\Sigma_\mu$ on $\En{\X}$ which consists of all entrance paths $\{(\mu(y_\bullet) > y_\bullet)\}$ is mild.
\end{proposition}
\begin{proof}
The first mildness condition from Definition \ref{def:mild} is implied trivially by the {\bf exhaustion} and {\bf order} axioms of Definition \ref{def:morsys} and the fact that there are only finitely many cells. Turning to the second mildness condition, let $f = (x > y)$ be any element of $\Sigma_\mu$. Now, $\En{\X}|f$ is the full subcategory of $\En{\X}$ generated by all the faces of $x$ except $y$. Its finiteness and loopfreeness are straightforward to confirm, so we focus here on proving contractibility of the classifying space. Note that $\En{\X}|f$ is nonempty because by regularity $x$ must have at least one face different from $y$. Moreover, the subcomplex generated by these non-$y$ faces of $x$ is homeomorphic to (some finite regular cellulation of) a sphere of dimension $(\dim x - 1)$ which is missing the interior of a top-dimensional cell $y$. Thus, this subcomplex --- and hence, by Proposition \ref{prop:entpathhom}, its entrance path category $\En{\X}|f$ --- is contractible, as desired.
\end{proof}

\subsection{Contractibility of the under-fibers}

This section will conclude with a proof of the following result, which (in conjunction with Theorem \ref{thm:cellularloc}) establishes Theorem \ref{thm:main}.

\begin{theorem}
\label{thm:floloc}
If $\Sigma$ is a mild Morse system on the cellular category $\bE$, then the inclusion $\bJ:\Flo_\Sigma\bE \to \Loc_\Sigma\bE$ induces a homotopy-equivalence of classifying spaces.
\end{theorem}
Note from Proposition \ref{prop:critfiber} that fibers of $\bJ$ under critical cells are contractible regardless of mildness, so we focus on fibers under those objects which lie in the span of some morphism in $\Sigma$.  We proceed by induction over the partial order $\flord$, starting with its locally minimal elements and working our way upwards. 
\begin{proposition}
\label{prop:baseunder}
If $\Sigma$ is mild and if $f:x \to y$ is locally minimal, then the fiber $w \fiber \bJ$ has a contractible classifying space for every $w$ in the span $S(f)$.
\end{proposition}
\begin{proof}
We will construct a strict p-functor $\bG:w\fiber\bJ \to \bE|f$ with contractible fibers and use the contractibility of its target category (implied by the second mildness assumption from Definition \ref{def:mild}). To this end, we provisionally define $\bG$ as follows on an object $(m,\gamma)$ and a morphism $\rho:(m,\gamma) \to (m',\gamma')$  of $w \fiber \bJ$: 
\begin{itemize}
\item $\bG(m,\gamma) = m$, and 
\item $\bG(\rho) = \rho:m \to m'$. 
\end{itemize}
We now check that $m$ and $\rho$ must comprise an object and morphism in $\bE|f$ respectively. The existence of $\gamma:w \to m$ in $\Loc_\Sigma\bE$ implies that there is a $\Sigma$-zigzag from $w$ to $m$, say
\[
\xymatrixrowsep{0.3in}
\xymatrixcolsep{0.3in}
\xymatrix{
w \ar@{->}[r]^{} & y_0 & x_0 \ar@{->}[l]_{} \ar@{->}[r]^{} & y_1 & \cdots \ar@{->}[l]_{} &  x_k \ar@{->}[l]_{} \ar@{->}[r]^{} & m.
}
\]
Since $f$ is locally minimal, all left-pointing $y_\bullet \gets x_\bullet$ above must equal $f$. Therefore, $\bE(x,m) \neq \varnothing$ and so $m$ lies in $\bE|f$ (note that $\bE(m,y) = \varnothing$ is guaranteed by the criticality of $m$). Turning to $\rho$, note that every object of $\bE|f$ is $\Sigma$-critical --- otherwise, there exists some $f': x' \to y'$ different from $f$ with $\bE(x,y') \neq \varnothing$ and hence $f \flord f'$, which violates the local minimality of $f$. Thus, there are no non-trivial $\Sigma$-zigzags between $m$ and $m'$, which means $\Loc_\Sigma\bE(m,m') = \bE(m,m')$. Thus, $\rho$ lies in $\bE|f(m,m')$ as desired. Now, it is easily checked that the fiber $\bG\fiber m$ over any $m$ in $\bE|f$ has a homotopy-maximal element $(m,\gamma_m)$ given by
\[
\gamma_m = \begin{cases}
														w \to m & \text{ if } \bE(w,m) \text{ is non-empty},\\
														w \to y \gets x \to m & \text{ otherwise,}															
												\end{cases}
\]
where all arrows in sight indicate atoms in the appropriate posets (which must exist by Definition \ref{def:cell2cat}). By Lemma \ref{lem:qinit}, the fiber $\bG\fiber m$ is contractible; thus, by Theorem \ref{thm:quillenA} we have a homotopy-equivalence between the classifying spaces of $w \fiber \bJ$ and the contractible category $\bE|f$, as desired.
\end{proof}
And finally, we have the inductive step which completes our proof of Theorem \ref{thm:floloc}.
\begin{proposition}
\label{prop:floind}
Assume $\Sigma$ is mild and contains $f:x \to y$. If the fibers $z \fiber \bJ$ under all objects $z$ in $\bE|f$ are contractible, then so is the fiber $w \fiber \bJ$ for each $w$ in the span $S(f)$.
\end{proposition}
\begin{proof}
It suffices to show that $|\Delta(x \fiber \bJ)|$ is contractible, since the fiber $w\fiber\bJ$ under any $w \in S(f)$ is homotopy-equivalent to $\bE(w,y) \times x \fiber\bJ$ by arguments similar to those from the proof of Proposition \ref{prop:baseunder}. Every object $(m,\gamma)$ in $x \fiber \bJ$ is represented by a $\Sigma$-zigzag
\[
\xymatrixrowsep{0.3in}
\xymatrixcolsep{0.3in}
\xymatrix{
x \ar@{->}[r]^{g_0} & y_0 & x_0 \ar@{->}[l]_{f_0} \ar@{->}[r]^{g_1} & y_1 & \cdots \ar@{->}[l]_{f_1} &  x_k \ar@{->}[l]_{f_{k}} \ar@{->}[r]^{g_{k+1}} & m,
}
\]
and no generality is sacrificed by assuming that $f_0$ (and hence, all subsequent $f_\bullet$'s) are different from $f$. Now, $y_0$ lies in $\bE|f$ and the morphism $\gamma$ in $\Loc_\Sigma\bE(x,m)$ decomposes into an element $g_0 \times (m,\lambda)$ in the product $\bE(x,y_0) \times y_0\fiber \bJ$, where $\lambda$ is the following morphism in $\Loc_\Sigma\bE(y_0,m)$ obtained from $\gamma$ by replacing $g_0$ with a forward-pointing identity:
\[
\xymatrixrowsep{0.3in}
\xymatrixcolsep{0.3in}
\xymatrix{
y_0 \ar@{->}[r]^{1} & y_0 & x_0 \ar@{->}[l]_{f_0} \ar@{->}[r]^{g_1} & y_1 & \cdots \ar@{->}[l]_{f_1} &  x_k \ar@{->}[l]_{f_{k}} \ar@{->}[r]^{g_{k+1}} & m.
}
\]
Thus, the category $x \fiber \bJ$ is covered by finitely many categories $\{\bV[z]\}$, one for each object $z$ in $\bE|f$, defined as images of products:
\[
\bV[z] = \bE(x,z) \times z \fiber \bJ,
\]
and the covering maps $\phi_z:\bV[z] \to x\fiber\bJ$ act via composition: 
\[
\phi_z(g \times (m,\lambda)) = (m,g\com\lambda).
\]
It is straightforward to confirm that the images of these covering maps are contractible subcategories of $x \fiber \bJ$. Note that every object in a non-empty intersection $V[z] \cap V[z']$ corresponds to a morphism $h:z \to z'$ in the following sense. We have an equality
\[
\phi_z(g \times (\lambda,m)) = \phi_{z'}(g' \times (\lambda',m'))
\]
if and only if $m = m'$ and there is some $h:z \to z'$ in $\bE|f$ making the following diagram commute in $\Loc_\Sigma\bE$:
\[
\xymatrixrowsep{0.25in}
\xymatrixcolsep{0.4in}
\xymatrix{
& z \ar@{->}[dd]_{h} \ar@{->}[dr]^{\lambda} & \\
x \ar@{->}[ur]^{g} \ar@{->}[dr]_{g'} & & m \\
& z' \ar@{->}[ur]_{\lambda'}
}
\]
Thus, every non-empty intersection $\bV[z_0] \cap \cdots \cap \bV[z_\ell]$ is contractible --- it is given by
\[
\bE(x,z_0) \times \bE(z_0,z_1) \times \cdots \times \bE(z_{\ell-1},z_\ell) \times z_\ell \fiber \bJ,
\] 
where the last factor is contractible by assumption (on fibers under objects of $\bE|f$) and the remaining factors are contractible since they contain atoms. Note that there are no other orderings of the $z_\bullet$'s which generate objects in the nonempty intersection above since $\bE|f$ is loopfree by mildness. This cover by $\bV[z]$'s of $x \fiber \bJ$ has a nerve whose vertex set is $(\bE|f)_0$ with order relation $z \geq z'$ whenever $\bE(z,z') \neq \varnothing$. By (arguments similar to the ones used in) the proof of Proposition \ref{prop:entpathhom}, this nerve lies in the homotopy class of $\Delta(\bE|f)$, which is contractible by Definition \ref{def:mild}.
\end{proof}

Turning at last to the proof of Theorem \ref{thm:floloc}, let $w$ be an arbitrary object of $\bE$ (and hence, of $\Loc_\Sigma\bE$). If $w$ is $\Sigma$-critical, then the fiber $w \fiber \bJ$ is contractible by Proposition \ref{prop:critfiber}. Otherwise, we know that there are no infinitely long chains of the form $f \flord f_0 \flord \cdots$ by mildness of $\Sigma$ and hence one may use an inductive argument as follows. If there are no $f' \in \Sigma$ different from $f$ satisfying $f \flord f'$, then $f$ is locally minimal and Proposition \ref{prop:baseunder} guarantees the contractibility of $w \fiber \bJ$. If we do have such $f'$s on the other hand, then assume (as an inductive hypothesis) that the fibers under all objects lying in their spans are contractible. But each $z \in (\bE|f)_0$ is either critical or lies (uniquely) in the span of some such $f'$. In this case, we resort to Proposition \ref{prop:floind} to extract the desired contractibility of $w\fiber\bJ$. Thus, $|\Delta(w \fiber \bJ)|$ is contractible for every object $w$ in $\Loc_\Sigma\bE$ and hence by the under-fiber version of Theorem \ref{thm:quillenA} our functor $\bJ$ induces a homotopy-equivalence of classifying spaces, as desired.

\section{Three calculations} \label{sec:calc}

We describe three computations of the discrete flow category in the most familiar and motivating context of acyclic partial matchings on regular CW complexes. The underlying CW complex in all three cases is the decomposition of the 2-sphere $\Sph$ from Figure \ref{fig:cwsphere} consisting of two 0-cells $w,y$, two 1-cells $x,z$ and two 2-cells $t,b$. 

\subsection{The flow category of an acyclic partial matching}\label{calc1}

For our first computation, consider the following acyclic partial matching $\mu$ on $\Sph$: set $\mu(y) = x$ and $\mu(z) = b$ and let 
\[
\Sigma = \{(x > y),(z > b)\}
\] be the associated Morse system on the entrance path category $\En{\Sph}$. Since $t$ and $w$ are the only $\Sigma$-critical cells, the discrete flow category $\Flo_\Sigma\En{\Sph}$ has precisely two objects. And since the poset $\En{\Sph}(w,\bullet)$ is empty for all cells $\bullet$ different from $w$, there are no non-trivial $\Sigma$-zigzags which start at $w$, and in particular $\Loc_\Sigma\En{\Sph}(w,t)$ is empty. A brief examination of Definition \ref{def:geonerve} now reveals that the classifying space $|\Delta\Flo_\Sigma\En{\Sph}|$ is (homeomorphic to) the suspension\footnote{The (two-point) suspension of a topological space $X$ is the quotient of $X \times [0,1]$ by identifications of the form $(x,0) \sim (x',0)$ and $(x,1) \sim (x',1)$ for each $x \in X$.} of $|\Delta\Flo_\Sigma\En{\Sph}(t,w)|$. 

In order to compute $\Flo_\Sigma\En{\Sph}(t,w)$, we examine all $\Sigma$-zigzags from $t$ to $w$. The easiest ones to describe are those with no backward pointing arrows which lie in the (un-localized) poset $\En{\Sph}(t,w)$:
\begin{align}
\label{zz1}
(t > z > w) \twomorback (t > w) \twomor (t > x > w). 
\end{align}
Since $\Sigma$ contains $(x > y)$, we are allowed to introduce zigzags of the form $t \to y \gets x \to w$. These correspond to the poset product $\En{\Sph}(t,y) \times \En{\Sph}(x,w)$:
\begin{align}
\label{zz2}
(t > x > y < x > w) \twomorback (t > y < x > w) \twomor (t > z > y < x > w).
\end{align}
Since $\Sigma$ also contains $(b > z)$, we have two new classes of zigzags. The first class is of the form $t \to z \gets b \to w$, and (similar to the previous poset product) it is given by
\begin{align}
\label{zz3}
(t > z < b > x > w) \twomorback (t > z < b > w) \twomor (t > z < b > z > w),
\end{align}
while the second class involves the longer zigzags $t \to z \gets b \to y \gets x \to w$. This is the product of three posets:
\begin{align}
\label{zz4}
(t > z) \times \big{[} (b > x > y) \twomorback (b > y) \twomor (b > z > y)\big{]} \times (x > w).
\end{align}

In order to assemble these four pieces together, we simply make the identifications suggested by the vertical reductions from Remark \ref{rmk:reductions} --- so $x > y < x$ is just $x$ while $z < b > z$ reduces to $z$, and so forth. For instance, the left side of (\ref{zz1}) is identified with the right side of (\ref{zz3}), while the right side of (\ref{zz1}) coincides with the left side of (\ref{zz2}). Making all such identifications leaves the following poset:
\[
\xymatrixrowsep{0.13in}
\xymatrixcolsep{0.13in}
\xymatrix{
(t > z > w)  & (t > w) \ar@{=>}[l] \ar@{=>}[r] & (t > x > w) \\
(t > z < b > w)\ar@{=>}[d] \ar@{=>}[u] &   & (t > y < x > w) \ar@{=>}[u] \ar@{=>}[d]  \\
 (t > z < b > x > w)   & (t > z < b > y < x > w) \ar@{=>}[l] \ar@{=>}[r] &(t > z > y < x > w) 
}
\]
The classifying space of the poset above is clearly homeomorphic to the circle, and therefore its suspension recovers $\Sph$ (up to homeomorphism and hence homotopy type) as desired. Note that we may coarsen $|\Delta\Flo_\Sigma\En{\Sph}|$ into a (non-regular) CW complex consisting of a 0-cell $w$ and a 2-cell $t$, where the entire boundary of $t$, which has the homotopy type of a circle as shown above, is glued onto $w$. 

\subsection{The necessity of poset-enrichment}\label{calc2}

The reader might wonder why we resort to the relatively strenuous process of localizing p-categories. One could, for instance, try to construct the discrete flow category by simply localizing about the paired cells in the face poset. Our second calculation reveals that this approach fails even when dealing with the Morse system $\Sigma = \{(x > y), (b > z)\}$ as before, but now on $\Fc{\Sph}$ rather than $\En{\Sph}$. 

The calculation proceeds in a similar manner to the preceding one, but we must convert some order relations to equalities when constructing $\Flo_\Sigma\Fc{\Sph}(t,w)$.  For instance, instead of (\ref{zz1}) we have
\[
(t > z > w) = (t > w) = (t > x > w),
\]
and so on. This process should not be too mysterious: we have simply applied the projection functor $\bP:\En{\Sph} \to \Fc{\Sph}$ from the proof of Proposition \ref{prop:entpathhom} to (\ref{zz1})-(\ref{zz4}), so that only the strictly alternating paths remain. In particular, all entrance sub-paths of the form $(p > q > r)$ are reduced to the extremal face relation $(p > r)$. The poset $\Flo_\Sigma\Fc{\Sph}(t,w)$ therefore equals
\[
\xymatrixrowsep{0.1in}
\xymatrixcolsep{0.1in}
\xymatrix{
  & (t > w) &  \\
(t > z < b > w)\ar@{=>}[ur] &   & (t > y < x > w) \ar@{=>}[ul]   \\
  & (t > z < b > y < x > w) \ar@{=>}[ul] \ar@{=>}[ur] \ar@{=>}[uu] &
}
\]
Since this poset's classifying space is contractible, so is its suspension $|\Delta\Flo_\Sigma\Fc{\Sph}|$. Thus, the flow category does not recover the homotopy type of $\Sph$ in this case. One reason for this failure is that the {\bf switching} and {\bf lifting} axioms of Definition \ref{def:morsys} do not hold --- we have:
\[
(b > z > y) = (b > y) = (b > x > y),
\]
but $z \not> x$. In fact, the localization functor $\Fc{\X} \to \Loc_\Sigma\Fc{\X}$ already fails to induce homotopy-equivalence. 

\subsection{The flow category of a generalized acyclic partial matching}\label{calc3}

Recent work on equivariant discrete Morse theory \cite{freij} introduced generalized Morse matchings, which relax the {\bf incidence} requirement of Definition \ref{def:partmatch} --- given a pairing $\mu(a) = b$ of cells, one only requires that $a$ be a face of $b$ with no restrictions on dimension. In this more general context, a critical cell $z$ is one which does not satisfy $a < z < \mu(a)$ for any pair $\mu(a) > a$. The {\em cluster lemma} \cite[Lem 4.1]{hershopt} or \cite[Lem 2]{clusterlemma} implies that given a generalized acyclic matching on a regular CW complex, there exists a traditional acyclic matching (in the sense of Definition \ref{def:partmatch}) with the same set of critical cells\footnote{In fact, one can expect several traditional acyclic matchings to yield the same critical cells as a fixed generalized matching, and in particular there is no canonical candidate.}. In our third and final calculation, we employ a generalized acyclic matching on $\Sph$ and construct the discrete flow category.

Consider the Morse system $\Gamma = \{(b > y)\}$ on $\En{\X}$. Since both $x$ and $z$ lie in the span of $(b > y)$, only $t$ and $w$ are critical as before. Thus, one needs to compute $\Flo_\Gamma\En{\Sph}(t,w)$ in order to extract the homotopy type of $|\Delta\Flo_\Gamma\En{\Sph}|$. There are two types of zigzags to consider. First we have the trivial ones from (\ref{zz1}),  and then we have the zigzags of type $t \to y \gets b \to w$. This second type is given by the poset product
\[
\big{[} (t > x > y) \twomorback (t > y) \twomor (t > z > y) \big{]} \times \big{[} (b > x > w) \twomorback (b > w) \twomor (b > z > w) \big{]},
\]
which is slightly more formidable than the ones hitherto encountered:
\[
\xymatrixrowsep{0.2in}
\xymatrixcolsep{0.15in}
\xymatrix{
(t > x > y < b > x > w)  & (t > y < b > x > w) \ar@{=>}[l] \ar@{=>}[r] & (t > z > y < b > x > w) \\
(t >x > y < b > w)\ar@{=>}[d] \ar@{=>}[u] & (t > y < b > w) \ar@{=>}[l]  \ar@{=>}[r] \ar@{=>}[u] \ar@{=>}[d] \ar@{=>}[lu] \ar@{=>}[ld] \ar@{=>}[ru] \ar@{=>}[rd]  & (t > z > y < b > w) \ar@{=>}[u] \ar@{=>}[d]  \\
 (t > x > y < b > z > w)   & (t > y < b > z > w) \ar@{=>}[l] \ar@{=>}[r] &(t > z > y < b > z > w) 
}
\]

By the vertical reduction of Remark \ref{rmk:reductions}, we may identify $x > y < b > x$ with $x$ and similarly $z > y < b > z$ with $z$. So the top left corner of the poset above corresponds to the right side of (\ref{zz1}) whereas the bottom right corner identifies with the left side of (\ref{zz1}). Thus, the poset $\Flo_\Gamma\En{\Sph}(t,w)$ may be geometrically realized as a filled-in square along with an additional path connecting two vertices across a diagonal:
\begin{figure}[h!]
\includegraphics[scale=0.3]{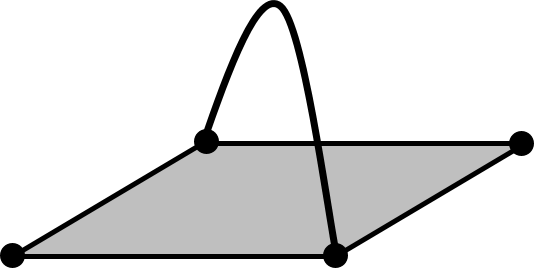}
\label{fig:genflo}
\end{figure}

This poset clearly has the homotopy type of a circle, and so its suspension $|\Delta\Flo_\Gamma\En{\Sph}|$ recovers the homotopy type of $\Sph$.

\section{The general main result and an application}\label{sec:newdmt}

Combining Theorems \ref{thm:cellularloc} and \ref{thm:floloc} yields the following general statement about (mild) Morse systems on cellular categories.
\begin{theorem}
\label{thm:gen}
Let $\bE$ be a cellular category equipped with a Morse system $\Sigma$, and let $\Flo_{\Sigma}{\bE}$ be the full subcategory of the localization $\Loc_\Sigma\bE$ generated by the $\Sigma$-critical objects. Then, the localization p-functor 
\[
\bL_\Sigma:\bE \to \Loc_\Sigma\bE
\] induces a homotopy-equivalence. If $\Sigma$ happens to be mild, then the inclusion functor 
\[
\bJ:\Flo_\Sigma\bE \to \Loc_\Sigma\bE
\] also induces a homotopy-equivalence, in which case the classifying spaces $|\Delta\bE|$ and $|\Delta\Flo_\Sigma\bE|$ are homotopy-equivalent.
\end{theorem}
Since every acyclic partial matching on a finite regular CW complex $\X$ induces a mild Morse system on its entrance path category $\En{\X}$ by Propositions \ref{prop:entcell} and \ref{prop:finmild}, our main Theorem \ref{thm:main} follows as an immediate corollary of the preceding result. Two independent avenues for generalizing discrete Morse theory are apparent from Theorem \ref{thm:gen}: we could impose Morse systems on cellular categories that are not entrance path categories, or we could examine CW complexes with Morse systems that are not induced by acyclic partial matchings. 

The calculation of the flow category from Section \ref{calc3} has already addressed the second type of generalization somewhat, so we will focus here on the first type. In particular, we use Theorem \ref{thm:gen} along with the universal property of localization to {\em conveniently inherit a Morse theory} for certain categories of functors sourced at $\En{\X}$. 

\subsection{Application: compressing cellular cosheaves}

Let $\X$ be a finite regular CW complex, and let $\Rmod$ be the category of modules over a fixed commutative ring $R$. Recall that the face poset $\Fc{\X}$ has as its objects the cells of $\X$ with a unique morphism $x > y$ whenever $y$ is a face of $x$. A {\bf cellular cosheaf} over $\X$ taking values in $\Rmod$ (see \cite[Sec 1.1]{shepard} or \cite[Sec 5]{curry}) is a functor $\bF:\Fc{\X} \to \Rmod$. The $\bF$-values on cells are called {\em stalks}, the linear maps assigned to face relations are called {\em extension maps}. One can compute the homology of $\X$ with coefficients in $\bF$ --- written $H_\bullet(\X;\bF)$ --- through a chain complex $(C_\bullet,d_\bullet)$ of $R$-modules:
\[
\cdots \stackrel{d_3}{\longrightarrow} C_2 \stackrel{d_2}{\longrightarrow} C_1 \stackrel{d_1}{\longrightarrow} C_0 \stackrel{d_0}{\longrightarrow} 0,  
\]
where $C_n =  \bigoplus_{\dim x = n} \bF(x)$ and the component of $d_n$ from $x$ to $y$ is the restriction map $\bF(x > y)$ multiplied with a suitable local orientation taking values in $\{-1,0,1\}$. 

Using the projection $\bP:\En{\X} \to \Fc{\X}$ from the proof of Proposition \ref{prop:entpathhom}, every cellular cosheaf $\bF$ on $\X$ induces a functor $\bP\com \bF:\En{\X} \to \Rmod$. Recall, by Proposition \ref{prop:finmild} that every acyclic partial matching $\mu$ on $\X$ induces a mild Morse system $\Sigma$ on $\En{\X}$. If all extension maps $\bF(\mu(\bullet) > \bullet)$ assigned to matched cells are isomorphisms of $R$-modules, then $\bP \com \bF$ admits an extension $\bF'$ across $\bL_\Sigma$ by the universal property of localization mentioned in Section \ref{subsec:loc}. Theorem \ref{thm:gen} now guarantees that the following diagram homotopy-commutes and that the vertical arrows are homotopy-invertible:
\[
\xymatrixrowsep{0.4in}
\xymatrixcolsep{0.4in}
\xymatrix{
\En{\X} \ar@{->}[dr]^{\bP \circ \bF} \ar@{->}[d]_{\bL_\Sigma}  & \\
\Loc_\Sigma\En{\X} \ar@{-->}[r]_{\bF'} & \Rmod \\
\Flo_\Sigma\En{\X} \ar@{->}[u]^{\bJ} \ar@{..>}[ur] &
}
\]
Here the dotted arrow is the composite $\bJ \com \bF'$ --- it yields a cosheaf of $R$-modules on the flow category, which in turn produces the Morse chain complex mentioned at the end of Section \ref{sec:DMT} and hence recovers the homology $H_\bullet(\X;\bF)$ (see Forman's work \cite[Sec 8]{forman98} for the simplest case, which involves the constant cosheaf). It is worth noting that the {\em multiplicity} of a gradient path in the sense of \cite[Def 8.6]{forman98} is precisely the action of $\bJ \com \bF'$ on any $\Sigma$-zigzag representing that path.

While there exist algebraic \cite{skoldberg} and computational \cite{curry:ghrist:nanda} techniques to extract the Morse chain complex for the purposes of computing $H_\bullet(\X;\bF)$, using the flow category allows us to simultaneously deform the underlying base space (while preserving its homotopy type) as we modify the overlaid algebra (while preserving homology). Thus, Theorem \ref{thm:gen} provides a natural mechanism to safely compress both the base space and the cosheaf data.

\section*{Acknowledgements}

This work benefitted greatly from the important ideas which Dai Tamaki and Kohei Tanaka shared with me while we worked on \cite{ntt}, and from Rob Ghrist's impeccable guidance and support. I first encountered Definition \ref{def:ent2cat} in a lecture series by Bob MacPherson and Amit Patel, which was organized by Justin Curry at the Institute for Advanced Study. I am also grateful to the anonymous referee for several insightful suggestions and corrections. This work was supported by the Alan Turing Institute under the EPSRC grant number EP/N510129/1. 

\bibliographystyle{abbrv}
\bibliography{morse}

\end{document}